\documentclass[12pt]{amsart}
\usepackage{latexsym,amssymb,amsmath,amsthm,graphicx,listings,comment}
\newtheorem{thm}{Theorem}
\newtheorem{propo}{Proposition}
\newtheorem{lemma}{Lemma}

\newtheorem{coro}{Corollary}
\title{Sphere geometry and invariants}

\author{Oliver Knill}
\date{February 12, 2017}
\address{Department of Mathematics \\ Harvard University \\ Cambridge, MA, 02138, USA }

\begin{document}

\maketitle

\begin{abstract}
A finite abstract simplicial complex $G$ defines two finite simple graphs: 
the Barycentric refinement $G_1$, connecting two simplices if one is a subset
of the other and the connection graph $G'$, connecting two simplices if they intersect. 
We prove that the Poincar\'e-Hopf value
$i(x)=1-\chi(S(x))$, where $\chi(S(x))$ is the Euler characteristics of the 
unit sphere $S(x)$ of a vertex $x$ in $G_1$, agrees with the Green 
function $g(x,x)=(1+A')^{-1}_{xx}$, where $A'$ is the
adjacency matrix of the connection graph $G'$ of the complex $G$. 
By unimodularity $\psi(G) = {\rm det}(1+A')=\prod_x (-1)^{\rm dim(x)} = \phi(G)$, 
the Fredholm matrix $1+A'$ is in ${\rm GL}(n,\mathbb{Z})$, where $n$ is the number of simplices in $G$.
We show that the set of possible unit sphere topologies in $G_1$ are combinatorial invariants of the 
complex $G$, and establish so that also the Green function range of $G$ is a combinatorial invariant.
The unit sphere character formula $g(x,x)=i(x)$ applies especially for the prime graph 
$G(n)$ and prime connection graph $H(n)$ on square free integers in $\{2, \dots,n\}$ playing the role
of simplices. In $G(n)$, integers $a,b$ are connected if $a|b$ or $b|a$ and where in $H(n)$ two $a,b$ 
are connected if ${\rm gcd}(a,b)>1$.
The Green function $g(x,x)$ in $H(n)$ relate there to the index values $i(x)$ in $G(n)$.
To prove the invariance of the unit sphere topology we use that all unit spheres
in $G_1$ decompose as $S^-(x) +  S^+(x)$, where $+$ is the join and $S^-$ is a sphere. 
The join renders the category $X$ of simplicial complexes into a monoid, where the empty complex 
is the $0$ element and the cone construction adds 1. The augmented Grothendieck group 
$(X,+,0)$ contains the graph and sphere monoids $({\rm Graphs}, +,0)$ and $({\rm Spheres},+,0)$.
The Poincar\'e-Hopf functionals 
$G \to i(G)=1-\chi(G)$ or $G \to i_G(x)=1-\chi(S_G(x))$ as well as the volume
are multiplicative functions on $(X,+)$. For the sphere group, both $i(G)$ as well as $\psi(G)$ are 
characters. The join $+$ can be augmented by a product $\cdot$ so that we have a
commutative ring $(X,+,0,\cdot,1)$ in which there are both additive and multiplicative primes 
and which contains as a subring of signed complete complexes $ \pm K_i$
isomorphic to the integers $(Z,+,0,\cdot,1)$. Both for addition $+$ and multiplication $\cdot$, 
the question of unique prime factorization appears open. 
\end{abstract}

\section{Preface}
A finite abstract simplicial complex $G$ has a Barycentric refinement $G_1$ which is the
Whitney complex of a graph. If $\chi$ denotes the Euler characteristic functional,
we identify the values $i(x) = 1-\chi(S(x))$ of unit spheres $S(x)$ in $G_1$ 
as Green function values $g(x,x)=(1+A(G'))^{-1}_{xx}$, where $A(G')$ is the adjacency matrix of
the connection graph $G'$ of $G$. The graph $G'$ has like $G_1$ the set of simplices in $G$
as vertex set but connects two if they intersect; in $G_1$, two simplices are connected 
if one is contained in the other. 
Having established that $1+A(G')$ is unimodular \cite{Unimodularity},
the matrix entries $g(x,y)$ became interesting. We can now look at the 
connection graph $G_1'$ of $G_1$. Its Green function values are again related to 
unit spheres in the Barycentric refinement $G_2$ of $G_1$. We compare the unit sphere topologies
in $G_1$ with the unit sphere topologies of $G_2$. We show here that in 
$G_2$, no new topologies appear. It follows that the sphere index spectrum, the 
set of Green function values is a combinatorial invariant of a simplicial complex 

\begin{figure}[!htpb]
\scalebox{0.2}{\includegraphics{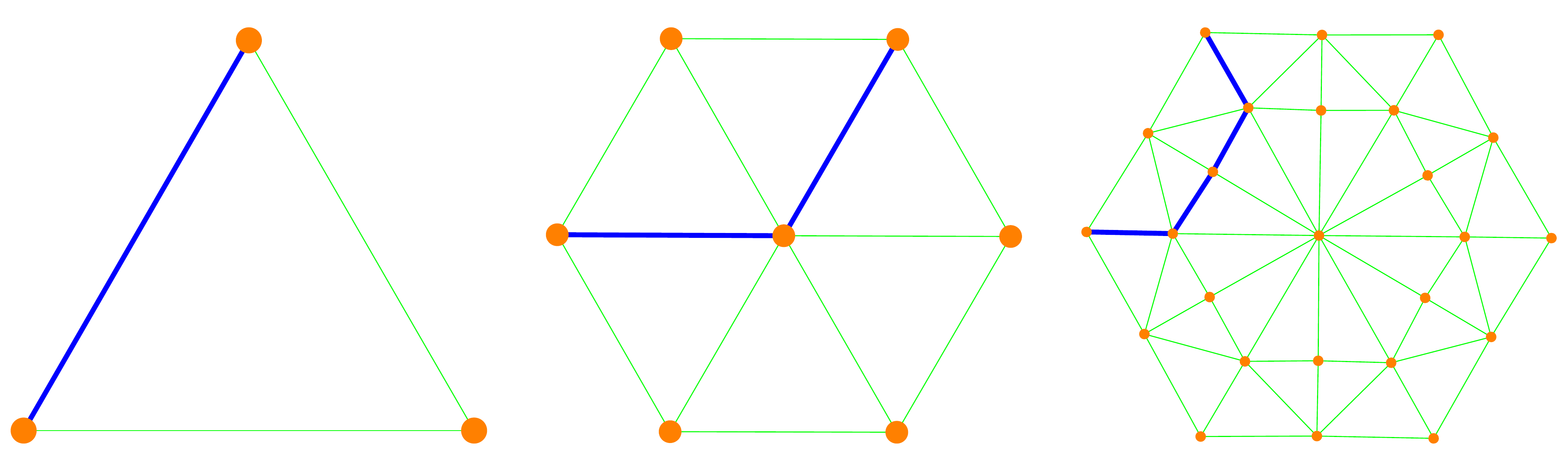}}
\caption{
We see the triangle $G$ and two of its Barycentric refinements $G_1$ and $G_2$. 
A choice of unit spheres in $G_1$ and $G_2$ are marked. 
\label{unitsphere1}
}
\end{figure}

\begin{figure}[!htpb]
\scalebox{0.2}{\includegraphics{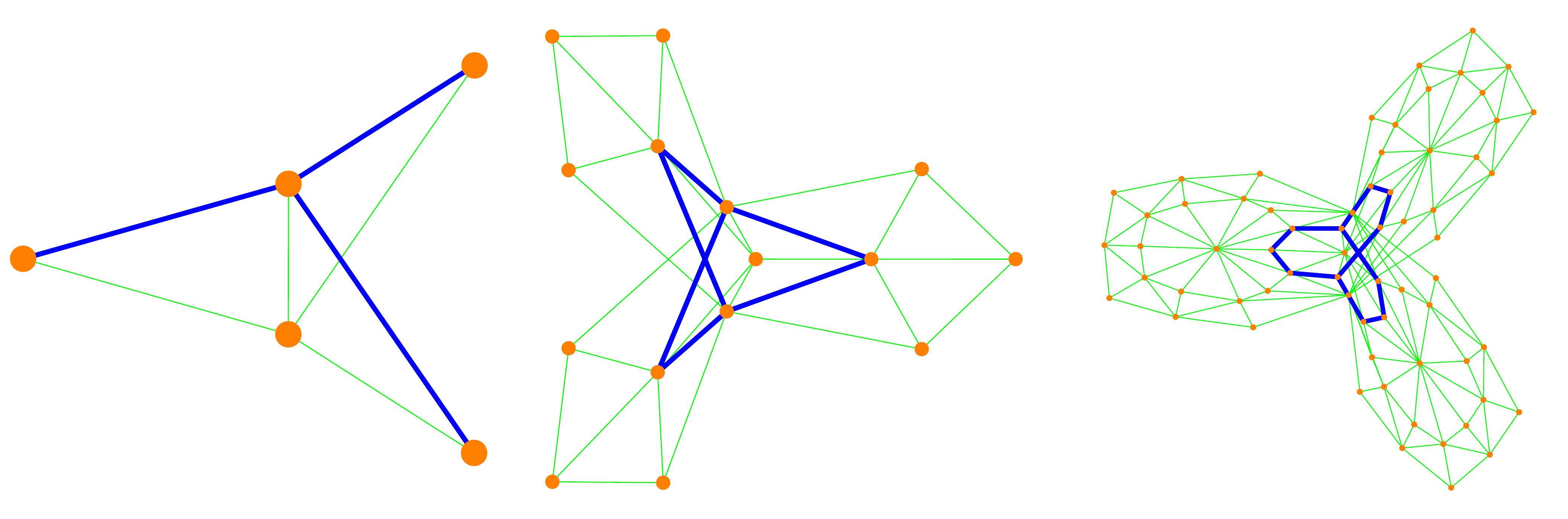}}
\caption{
For the windmill graph $G$, unit spheres
in $G_1$ can have topologies different from unit sphere topologies
in $G$. While all unit spheres in $G$ are contractible, there is a
unit sphere $S(x)$ in $G_1$ which as a suspension $P_2+P_3$ of the 3-point graph $P_3$
has $\chi(S(x))=\chi(P_2) + \chi(P_3)-\chi(P_2) \chi(P_3)=-1$. 
But in $G_2$, no new topologies appear in unit spheres.
\label{unitsphere2}
}
\end{figure}

\section{Introduction}

The quest to find invariants of topological spaces
is particularly concrete when searching for combinatorial invariants in
abstract finite simplicial complexes. These are quantities which do not change when applying a
Barycentric refinement \cite{Bott52}. Abstract simplicial complexes have
a surprisingly rich geometry despite the fact that they are one of the simplest objects
mathematics knows. Every partially ordered set for example defines a complex, the 
order complex. Besides matroids also
graphs are a source for complexes, like subcomplexes of the clique complex 
or then graphic matroids. Other type of simplicial complexes were introduced by
Jonsson in \cite{JohnsonSimplicial}. 
Graph theory enters naturally as any Barycentric refinement $G_1$ of an 
arbitrary abstract simplicial complex is already is the Whitney clique complex of a graph. The faces 
of $G$ are the vertices of $G_1$ and two faces connected if one is contained in the other.
The complex $G$ also defines the connection graph $G'$ on the same vertex set, 
where two simplices connect if they intersect. 
Both $G_1$ and $G'$ produce aspects of the geometry on $G$ which are well accessible as graphs
are not only intuitive, they also serve well as data structures.  \\

Examples of combinatorial invariants of simplicial complexes are
simplicial cohomology, Euler characteristic, homotopy groups (discrete notions of spheres 
and discrete notions of homotopy allow them purely combinatorically),
the Bott invariants \cite{Bott52}, the clique number (which is $1$ plus the maximal dimension), 
minimal entries of a $f$-vector \cite{Lutz1999,Lutz2005},
Wu characteristic, as well as connection cohomology attached to Wu characteristic \cite{Wu1953,valuation}
or minimal possible entries in the $f$-matrix. 
Examples of quantities which are not combinatorial invariants are the $f$-vector itself (as it
gets multiplied by a fixed upper triangular matrix when applying a refinement), 
the $f$-matrix, telling about the cardinalities pair intersections, 
the chromatic number (it can decrease under refinements but will 
stabilize already after one step to the clique number) or dimension (it can increase 
under refinements and will converge to the maximal dimension when iterating the Barycentric
refinement process), neither is the Fredholm characteristic $\psi$ 
(as it stabilizes to $1$ already after one Barycentric refinement).  \\

\begin{figure}[!htpb]
\scalebox{0.8}{\includegraphics{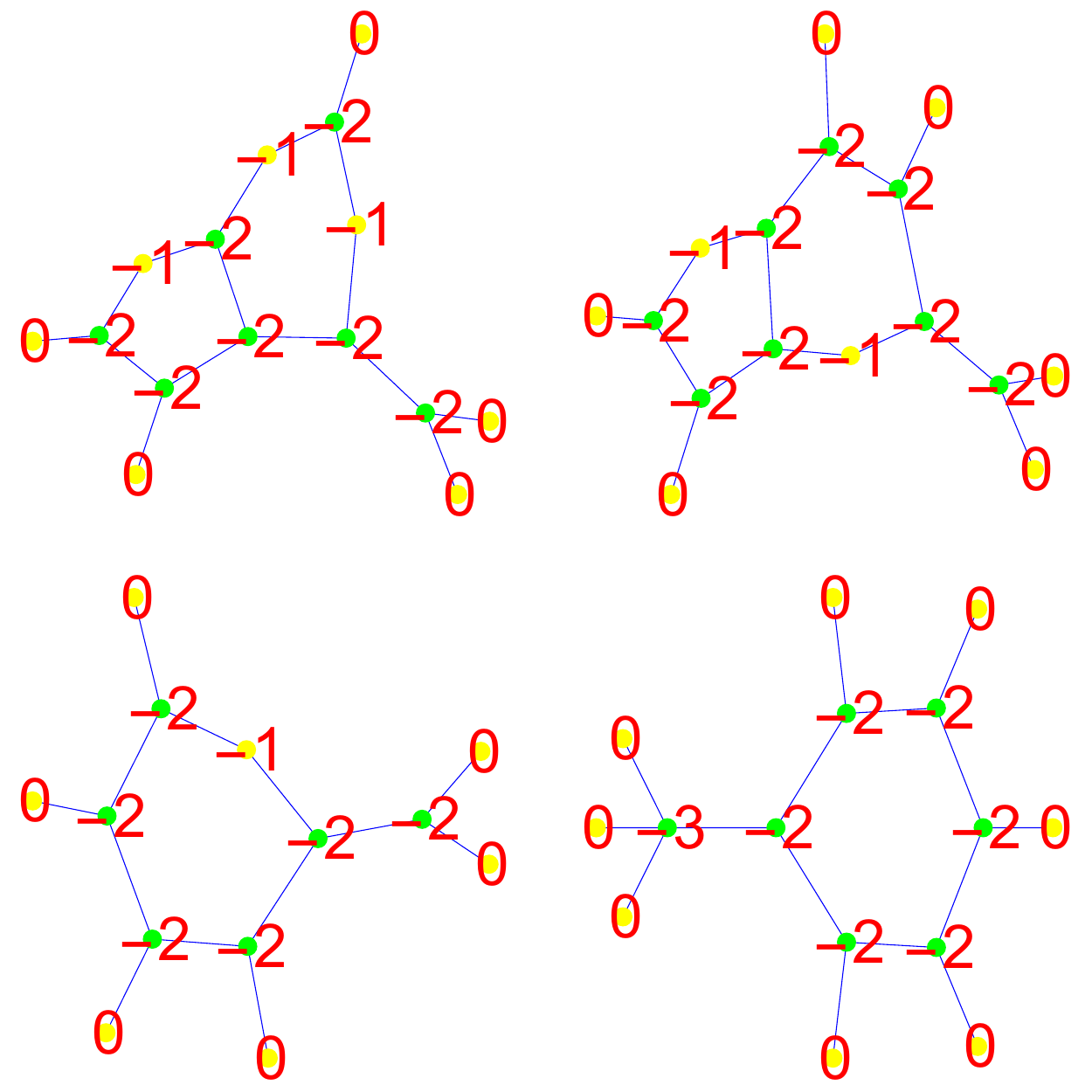}}
\caption{
Adenine, Guanine, Cytosine and Thymine graphs with 
sphere spectrum. These are all $1$-dimensional 
graphs. The Fredholm Characteristic $\psi$ is $1$ for Adenin
and $-1$ for the others. The Euler characteristic $\chi$ is
$-1$ for Adenine and Guanine and $0$ for Cytosine and 
Thymine. The sphere spectral values $\sigma$ depend only on 
vertex degrees as for all $1$-dimensional graphs. It is 
$\sigma=\{-2,-1,0\}$ for the first three.
Thymine alone has a different sphere spectrum $\sigma=\{-3,-2,0\}$. 
The values $\psi,\phi,\sigma$ are different for all. 
Also the $f$-vectors $(15,16),(16,17),(13,13),(15,15)$ differ. 
All four graphs are prime in the Zykov monoid of graphs. 
\label{dna}
}
\end{figure}

Focusing on finite simplicial complexes $G$ allows to do geometry on finite sets but
still be close to classical topology: the topology of a smooth compact manifold $M$ for example is completely
determined by the topology of the simplicial complex $G$ defined by a triangulation of $M$ in the form of a $d$-graph
and the abstract Whitney complex describing it contains all topological information about $M$.
While it is possible to look at topological invariants by embedding the complex into a continuum
or by realizing it as a polytop in an Euclidean space, we prefer here to look at 
invariants defined and computed in a finite combinatorial manner. This is 
pragmatic computer science point of view. \\

Both the set of abstract finite simplicial complexes as well as the set of finite simple graphs form a 
Boolean lattice: one can intersect and take unions.
They both do not form a Boolean ring however as the Boolean addition,
the symmetric difference operation, $A \Delta B = A \cup B \setminus A \cap B$ is no more in the
same category. A valuation is a numerical quantity $X$ satisfying $X(G \cup H)  + X(G \cap H) = X(G) + X(H)$.
If we had a Boolean ring, a valuation would satisfy the linearity condition $X(G \Delta H) = X(G) + X(H)$. 
By discrete Hadwiger, any valuation is of 
the form $X(G) = f(G) \cdot \chi$, where $\chi$ is a fixed vector and $f(G)$ is the $f$-vector of $G$. 
Since the only eigenvector of the Barycentric operator $f(G) \to f(G_1)$ is $(1,-1,1,-1, \dots)$
Euler characteristic is the only combinatorial invariant which is also a valuation. While Euler
characteristic $\chi(G) = \sum_x \omega(x)$ with $\omega(x)=(-1)^{{\dim}(x)}$ is pretty unique among 
valuations, there are other functionals. One is the Fermi characteristic $\psi(G) = \prod_x \omega(x)$,
a multiplicative cousin of Euler characteristic.  \\

Then there are multi-variate versions like the Wu characteristic $\omega(G) = \chi \cdot F(G) \chi$, 
where $F(G)$ is the $f$-matrix $F(G)_{ij}$ counting the number of intersections of $i$ and $j$-dimenensional 
simplices in the complex. 
Also the Wu characteristic as well as higher order versions are combinatorial invariants. 
A quadratic valuation $X(G,H)$ is a map for which $G \to X(G,H)$ and $H \to X(G,H)$ are both
valuations. Related to this had been a question of Gruenbaum \cite{Gruenbaum1970}, 
whether higher order Dehn-Sommerville relations exist. 
We answered this affirmatively in \cite{valuation}. The analysis there shows that Wu characteristic is
the only combinatorial invariant among multi-linear valuations. Connection calculus leads to other
combinatorial invariants like the Betti numbers of connection cohomology. That this is useful has already
been demonstrated in examples. It allows to distinguish the cylinder and the Moebius strip for example. \\

Quantities satisfying the multiplicative valuation property 
$X(G \cup H)$ $X(G \cap H)$ $= X(G) X(H)$ must
be additive valuations after taking logs and indeed, the Fredholm determinant is an example.
It is related to the valuation $f(G)$ counting the number of odd-dimensional simplices in $G$. \\

We should point out that cohomology with real-valued cocycles is well defined also for a finitist who does not
accept infinity. The reason is that the cohomology groups are determined by Hodge as
the nullity ${\rm dim}({\rm ker}(L_k))$ of finite integer-valued matrices $L_k$, the form Laplacians
$L_k$. The dimension of the kernel as well as the kernel basis itself can be obtained by row reduction, 
which is a finite process and done on integer matrices. So, cohomology is an acceptable concept for a 
finitist. For more about the complexity, see \cite{Joswig2004}.
This happens already when following the lead of Kirchhoff, Betti and Poincar\'e. 
In topology however
still, Euclidean realizations are used even for definitions like Barycentric refinements. But one
perfectly can stay within the realm of simplicial complexes or better in the realm of 
much the more accessible graph category, on which the complex is a subcomplex of the Whitney clique complex. \\

How do simplicial complexes relate to traditional topology? The usual approach is to realize the
complex in some sufficiently high dimensional Euclidean space and call two complexes $G,H$ topologically 
equivalent if their topological realizations $\overline{G},\overline{H}$ are
homeomorphic in the classical point set topology sense. 
A topological invariant of the complex is then a quantity
produces homeomorphic geometric realizations. For abstract finite
simplicial complexes, the notion of topological invariant is equivalent to the notion of 
combinatorial invariant but the later works in the more restricted axiom system
of finite mathematics. \\

The quest to find a complete finite set of invariants which allow to distinguish an arbitrary pair
of simplicial complexes has long shown to be unattainable: it is not possible to build a Turing machine 
which can distinguish any pair of abstract finite simplicial complexes. Markov used in 1958 the 
unsolvability of the word problem by Novikov to show that the homeomorphism problem 
is not solvable. For surveys, see \cite{Stillwell80,Andrews2005,Poonen2014}. \\

A constructivist would insist on giving a concrete homeomorphism which establishes the topological equivalence, 
but this is not so easy, as the Euclidean realizations are piecewise linear pieces of simplices, glued 
together in some high dimensional space. 
Also the complexity to decide whether a given pair of finite simplicial complexes are topologically equivalent
can be tough. Given two concrete complexes, how do we (by just dealing with finite sets) decide whether they
are topologically equivalent? We need other morphism for simplicial complexes in order to mirror
topological equivalence. One answer is to use a finite Zariski type topologies which is purely combinatorial.
It is conceivable however that we have to go through an exponentially large set of 
possible admissible finite topologies and check in each case the graph isomorphism problem. \\

Let us assume that the complex is the Whitney complex of a graph. As mentioned already, this is not much of a loss
of generality as a Barycentric refinement of an abstract simplicial complex is always the Whitney complex of a graph. 
Since a graph has a natural metric, the geodesic distance, one could see such a complex as an example of a 
finite metric space. But since every singleton set $\{x\}$ is both open and closed in that metric,
the topology is the discrete topology and not that interesting. A weaker topology is Zariski like: it the 
topology on the vertex set of $G_1$ defined in such a way that the vertex set of $H_1$ of any subcomplex $H$ of $G$ is closed.
This defines a finite topology on the vertex set of $G_1$, (which is in general not Hausdorff similarly as
Zariski topology), the homeomorphisms are still the graph isomorphisms. 
(The analogy is that subgraphs of $G$ play the role of varieties and homeomorphism the 
role of regular maps. The algebraic nature of the ``varieties" comes from the fact that we only 
allow subgraphs of $G$ and not subgraphs of $G_1$. The later would give the discrete topology.) \\

The Zariski idea is more flexible as
we can use the same definition to form much weaker topologies by starting with a smaller 
set of subgraphs of $G$ defining a sub-base of the topology. For example, we can get topologies on 
the second refinement $G_2$ by taking the set of graphs $H_{2}$ as closed, where $H$ is a subgraph of $H$. 
We have proposed \cite{KnillTopology} to insist that a finite topology on a graph should be given 
by a subbase and have the property that 
the nerve graph of a sub-base should be homotopic to the graph and that a dimension condition is 
satisfied for intersections. The reasons are that we want any reasonable notion of cohomology to 
agree with the natural topology of the graph and insist that any notion of homeomorphism should
honor the concept of dimension. \\

Unlike the quest to find a complete set of invariants for a simplicial complex,
the task to find concrete combinatorial invariants of an abstract simplicial complex is more accessible.
It is part of what we do here. An example of quantities which are interesting are valuations satisfying
$X(G \cup H) + X(G \cap H) = X(G) + X(H)$. 
The discrete Hadwiger theorem \cite{KlainRota} has classified 
all valuations on a complex and identified Euler characteristic as the only combinatorial 
invariant among them.  
It states that if the maximal dimension of $G$ is $d$, the space of valuations is $d+1$ dimensional.
More invariants can be obtained by looking a multi-linear valuations 
or multiplicative invariants $X(G \cup H) X(G \cap H) = X(G) X(H)$ like Fredholm characteristic. \\

Since Fredholm characteristic ${\rm det}(1+A(G'))$ 
is $\{-1,1\}$-valued for simplicial complexes we can
look at the Green function values $g(x,y) = (1+A')^{-1}_{xy}$. Investigating these integers 
led to the current paper. It is here related to locally defined combinatorial invariants,
the sphere Euler characteristic spectrum, which is defined by the 
collection of indices $i(S(x)) = 1-\chi(S(x))$ which unit spheres $S(x)$ in the complex can have. 
One could look at higher order versions like the quadratic sphere spectrum
$1-\chi(S(x) \cap S(y))$, where $x,y$ run both over all the simplices in $G$. This
however does not match the off-diagonal Green function values in general and the off-diagonal values $(1+A')^{-1}_{xy}$
with $x \neq y$ remain at the moment still unidentified. Analogies from physics suggest that all Green function values
should have some natural interpretation and are possibly of a dynamical nature. 
Work like \cite{ChungYau2000} show how close the discrete case can be to the continuum. 
In the shifted Fredholm case, it is exciting that 
the Green function values are integers and quantized. There would be other numbers 
to consider like $(1+A')^{k}_{xy}$ but we seem only get topological invariants for $k=-1$. \\

In the context of doing arithmetic on graphs, the proof uses join operations.
In \cite{HararyGraphTheory} pp.21 the definition is attributed to
A.A. Zykov who introduced it in 1949 \cite{Zykov}. There should be no confusion with the notion
of ``join" used for partially order sets and especially for simplicial complexes \cite{RhodesSilva}.
Topologically, the Zykov join of two graphs has the same properties as in the continuum 
and which appear in textbooks like \cite{RourkeSanderson,Hatcher}.
In the discrete, the construction does not need to take any quotient topologies: the vertex set of the 
join of two graphs is the union of the vertex sets and the edge set is the union of the edge sets of the factors together
with all pairs belonging to different graphs. The zero element $0$ is the empty graph. As in the continuum, the 
sum of two spheres is a sphere again. Adding $1$ to a graph $G+1$ is the cone construction.  \\

Unlike for the arithmetic
of numbers, where $1$ is the only additive prime, there are more additive primes in the join monoid. 
The $0$-dimensional sphere $P_2$ for example is prime. Adding such a sphere to a graph is the suspension. 
The maximal dimension of a sum $G+H$ is the sum of the maximal dimensions plus $1$. 
But there are still mysteries about the Zykov monoid. Unlike in the continuum, we suspect (but do not know nor dare to 
conjecture) that in the discrete, the join operation is a unique factorization monoid, where the empty graph plays the
role of the $0$ element and where the 1-point graph $K_1$ is the smallest example of an additive prime. The complete graph $K_n$ 
decomposes as $K_1 + \dots + K_1$ and $C_4$ can be written as $P_2 + P_2$. We can how for example that a triangulation of a 2d-surface of 
positive genus is prime, as the only 2d-surfaces which can be factored are of the form $C_n + P_2$ which are all 
2-spheres which are prism graphs for $n \geq 2$.  \\

By looking at the numerical quantity of maximal dimension, we can see that every graph $G$ factors additively 
into primes $G=p_1 + p_2 \dots + p_k$, but we don't know yet whether the factorization is unique. 
[I could not get hold yet of the Zykov article \cite{Zykov}. The MathSciNet review of Tutte
mentions a unique prime factorization result there. ]
We have infinitely many prime graphs, like graphs $P_n$ without edges, circular graphs $C_n$ 
with $n>4$ or disconnected graphs. A spectral condition is given at the end of this article.
In some sense, our understanding of this arithmetic is on the level of Euclid, 
who did not prove the fundamental theorem of arithmetic yet. But unique prime factorization might not 
even hold in the sphere sub-monoid within that graph monoid. The sphere monoid might be the easier one to attack.
But already there, for each positive dimension, there are infinitely many prime spheres: the number of maximal
simplices of $G+H$ is the product of the number of simplices of $G$ and $H$.
A sphere with a prime number of maximal simplices therefore is an additive
prime in the sphere monoid. In general, the Euler generating functions
$f_G(x)= 1+\sum_{k=0} v_k x^{k+1}$ for the join satisfies $f_{G+H} = f_G f_H$, which has as a corollary
$\chi(G+H) = \chi(G) + \chi(H) - \chi(G) \chi(H)$ which follows from $\chi(G)=1-f_G(-1)$. This formula
implies that $i(G)=1-\chi(G)$ is multiplicative and 
that graphs with zero Euler characteristic, or graphs with even Euler characteristic and 
graphs with odd Euler characteristic all form sub-monoids. \\

At the end we point out that there is a multiplication on simplicial complexes which is compatible with the 
Zykov join addition in the sense that distributivity holds. 
If the Zykov join monoid is extended to become a group
we get so a commutative ring of simplicial complexes. This ring extends naturally the ring of integers as
it contains the ring of integers in the form of complete graphs: $K_n + K_m = K_{n+m}$ and $K_n K_m = K_{n m}$. 
The empty graph $0$ is the zero element and the graph $K_1$ is the $1$-element. We originally looked for a 
ring structure in order to proof the unimodularity theorem. A purely algebraic proof has not worked yet and 
might not even exist.  
Indeed, the Fredholm functional $\psi$ is only multiplicative on the additive subgroup of 
signed complexes which have even Euler characteristic. 
But $\psi$ is a character on the sphere subgroup of the additive group similarly as 
the Poincar\'e-Hopf functional $i(G) = 1-\chi(G)$ which is a character on the sphere group and which 
plays here an important role. 

\section{Simplicial complexes}

A finite abstract simplicial complex is a finite set $G$ of non-empty sets, called simplices, 
so that $G$ is invariant under the operation of taking non-empty subsets. 
The dimension of a simplex $x$ in $G$ is defined as the cardinality of $x$ minus $1$. 
In graph settings, the cardinality is called the clique number. 
The maximal dimension of $G$ is the maximal dimension, which a simplex in $G$ can have. 
An abstract simplicial complex is sometimes also called hereditary collection
\cite{RhodesSilva}. Early references using finite abstract simplicial 
complexes are \cite{Spanier,KahnSaksSturtevant} not using Euclidean realizations exclusively
as in \cite{Wallace1957}. \\

Having only the subset axiom, simplicial complexes are one of the simplest
geometrical structures imaginable, simpler even than algebraic structures, a topology or
measure theoretical structures. Indeed, an abstract simplicial complex is an
order structure on a set and every partial ordered set defines a simplicial complex, its
order complex. Abstract finite simplicial complexes are therefore closely related to 
finite posets. In the rest of the paper, we often just say ``complex" meaning finite abstract
simplicial complex. \\

The Barycentric refinement $G_1$ of a complex $G$ is the set of subsets of the simplex sets $A$
which have the property that for any pair $(a,b)$ in $A$, either $a$ is a subset of $b$ or
$b$ is a subset of $a$. This defines a graph $G_1=(V_1,E_1)$ where $V_1$ is the set of simplices
in $G$ and $E_1$ the set of pairs $(a,b)$ such that $a$ is contained in $b$ or $b$ is contained
in $b$. The Whitney complex of this graph is then the complex $G_1$. One can also look at
the connection complex $G'$ of $G$ which is the Whitney complex of the graph where
two simplices $(a,b)$ are connected, if one is contained in the other. It can be easier to work with
graphs $G_1$ and $G'$ however rather than the intrinsic simplicial complexes. \\

{\bf Remarks.} \\
{\bf 1)} There is a more sophisticated dimension of a complex: the inductive dimension
${\rm dim}(x)$ of a simplex $x$. It is defined as $1+{\rm dim}(S(x))$, where $S(x)$ is the 
Whitney complex in the unit sphere $S(x)$ of $x$ in $G_1$, and
${\rm dim}(G) = (1/|V(G_1)|) \sum_x {\rm dim}(x)$. Unlike the maximal dimension, the 
inductive dimension is a rational number in general. It has some nice properties like
satisfying the same dimension inequality than Hausdorff dimension or having computable 
expectation on Erd\H{o}s R\'enyi probability spaces \cite{randomgraph}. \\

{\bf 2)} The Barycentric refinement of $G$ can be written as a product $G \times K_1$
\cite{KnillKuenneth}:
this simplicial product graph $G \times H$ of two complexes is the 
graph for which the vertices are the ordered pairs $(x,y)$ of simplices $x$ in $G$ and $y$ in $H$ 
for which two $(a,b),(c,d)$ are connected if either $a \subset c, b \subset d$ or
$c \subset a, d \subset b$. Even if $G$ is a one-dimensional graph, it has no relation with the 
Cartesian product or tensor product for graphs which both are graphs with vertex set $V(G) \times V(G)$. The
simplicial product graph $G \times H$ has a vertex set the product of the simplex sets of $G$ and $H$.
The simplicial product graph can be rewritten as the product in the Stanley-Reisner ring. 
It is a natural product as the Euler characteristic $\chi(G \times H) = \chi(G) \times \chi(H)$ 
is multiplicative and the K\"unneth formula for cohomology applies \cite{KnillKuenneth}. \\

Given a simplicial complex $G$, we can look both at the Barycentric refinement $G_1$ of $G$
as well as the connection graph $G'$ of $G$. Both graphs will play a role here. 
The two graphs $G'$ and $G_1$ have the same vertex set as this is the set of simplices of $G$. The graph
$G_1$ is a subgraph of $G'$. In the theory of simplicial complexes one looks also at the Hasse diagram, 
which is a subgraph of $G_1$. \\

The unit spheres $S(x)$ in $G_1$ have topologies which 
are interesting. It turns out that the possible topologies which appear as unit spheres in $G_1$
form a combinatorial invariant of $G$. While $G_1$ can feature new unit sphere topologies, 
there are no new sphere topologies appearing in $G_2$.  \\

Note that we work with general simplicial complexes and do not assume any Euclidean structure. 
Still, there is intuition which leads to the invariance of the unit sphere topology. 
A Barycentric refinements can produce new sphere types but after having done one refinement, the
local sphere geometry does not change any more in the interior. The inside of refined simplices is now
Euclidean. The unit spheres in the interior are discrete spheres and space looks there like Euclidean
space. It is getting more interesting at singularities, where unit spheres are no topological spheres
any more. An example is the figure $8$ graph. 
Most points in that graph have a zero-dimensional sphere as a unit sphere but there is a point, the singularity, 
where a sufficiently small sphere consists of 4 points. 
The value of the Euler characteristic is $4$. It is a value does not change under Barycentric refinement. \\

\begin{figure}[!htpb]
\scalebox{0.35}{\includegraphics{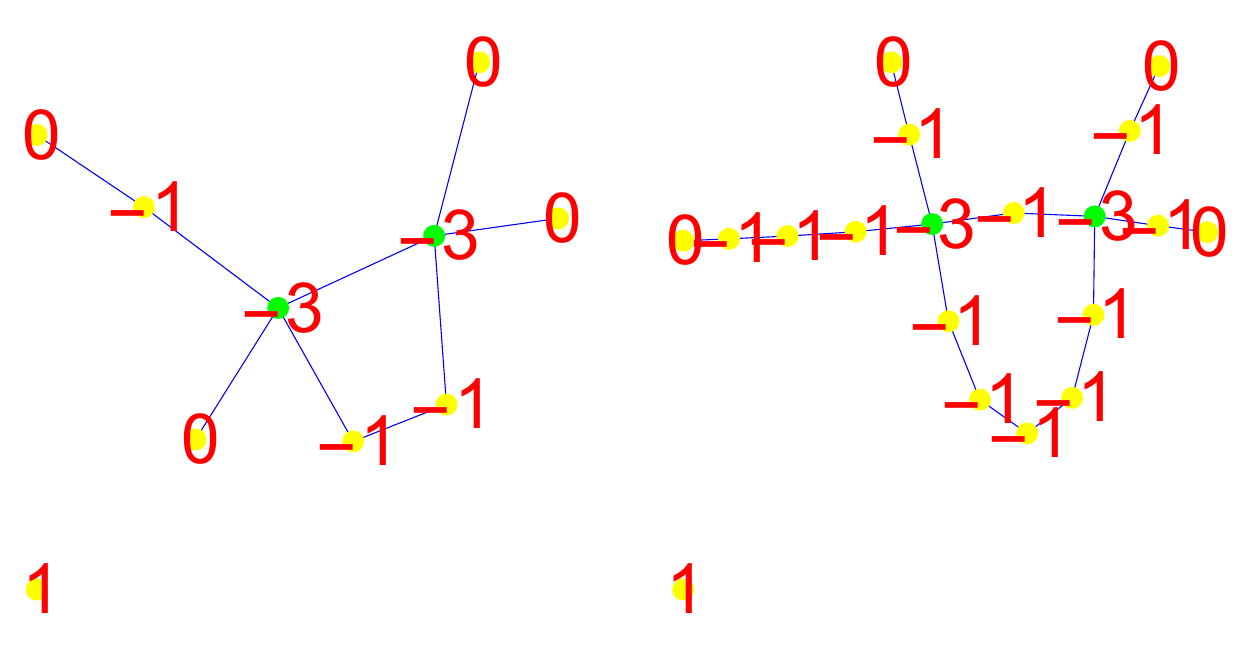}}
\scalebox{0.35}{\includegraphics{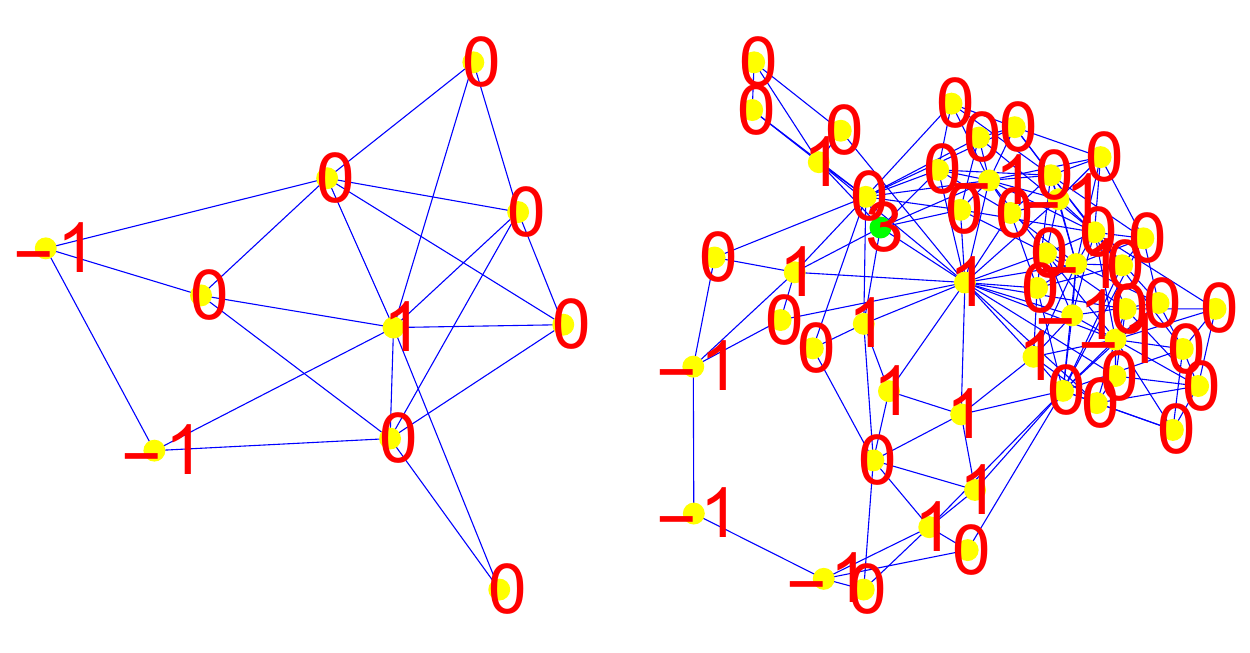}}
\scalebox{0.35}{\includegraphics{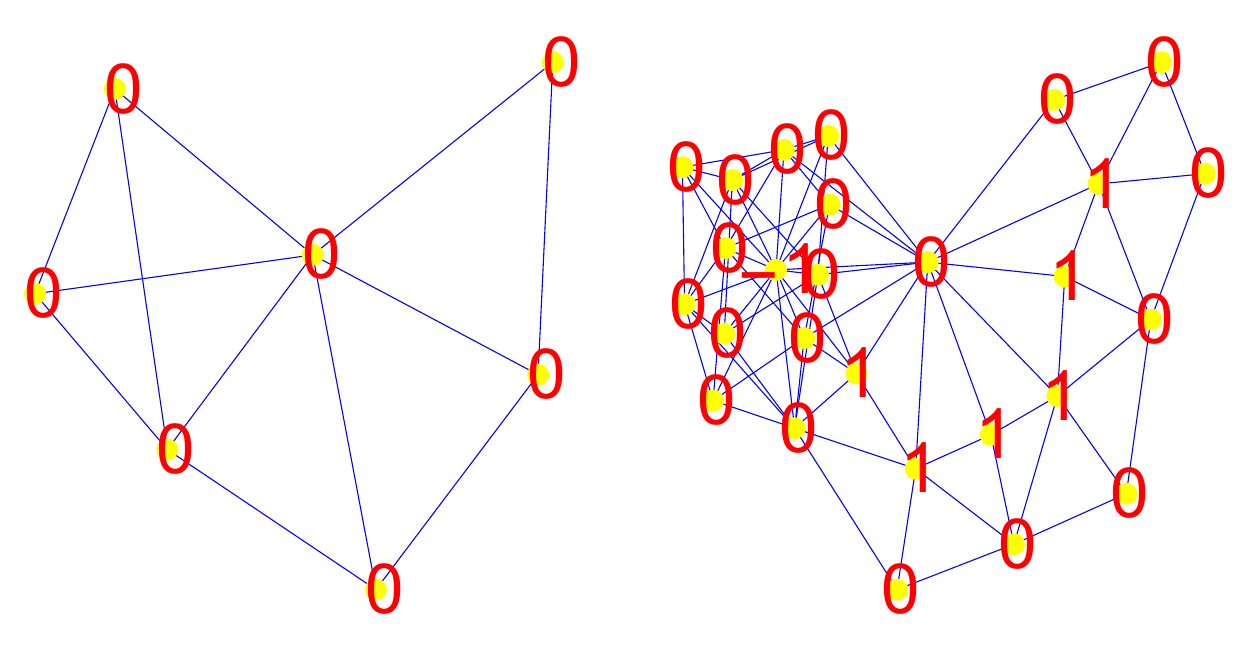}}
\scalebox{0.35}{\includegraphics{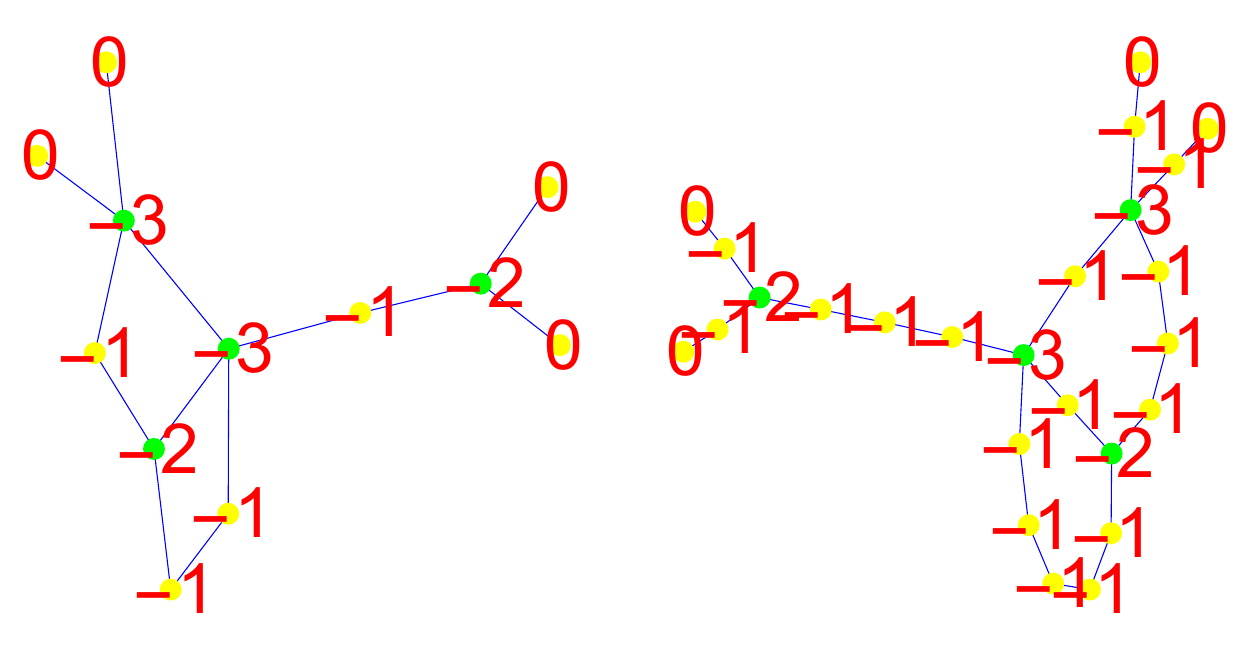}}
\scalebox{0.35}{\includegraphics{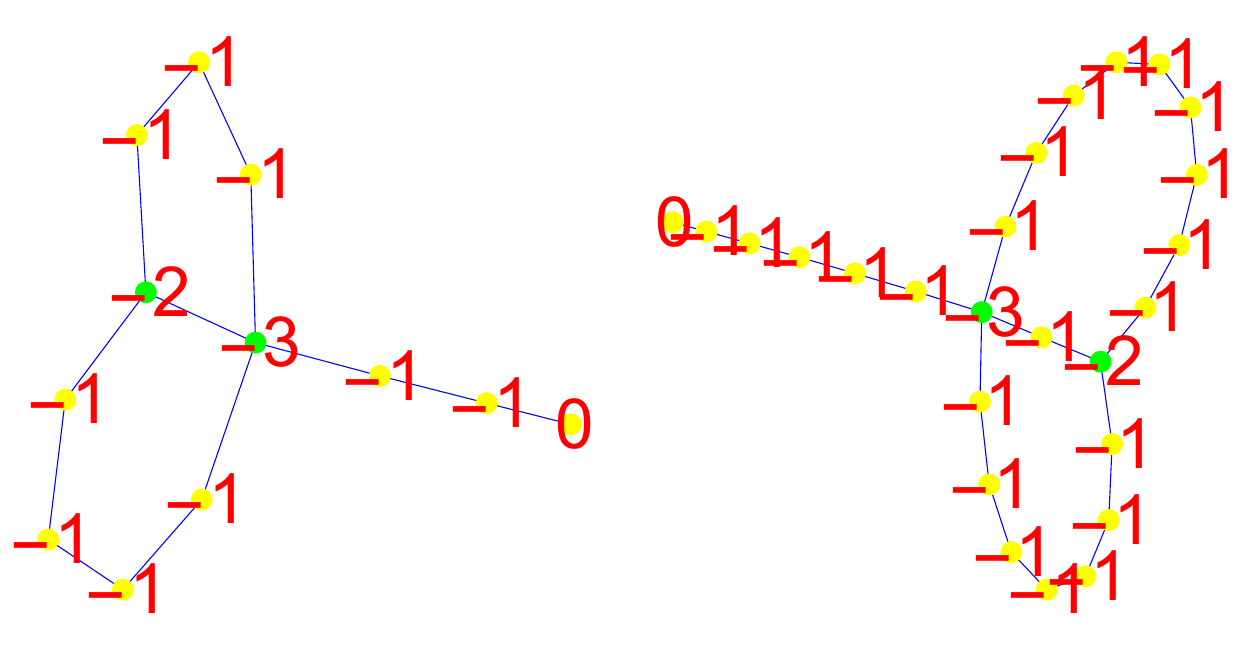}}
\scalebox{0.35}{\includegraphics{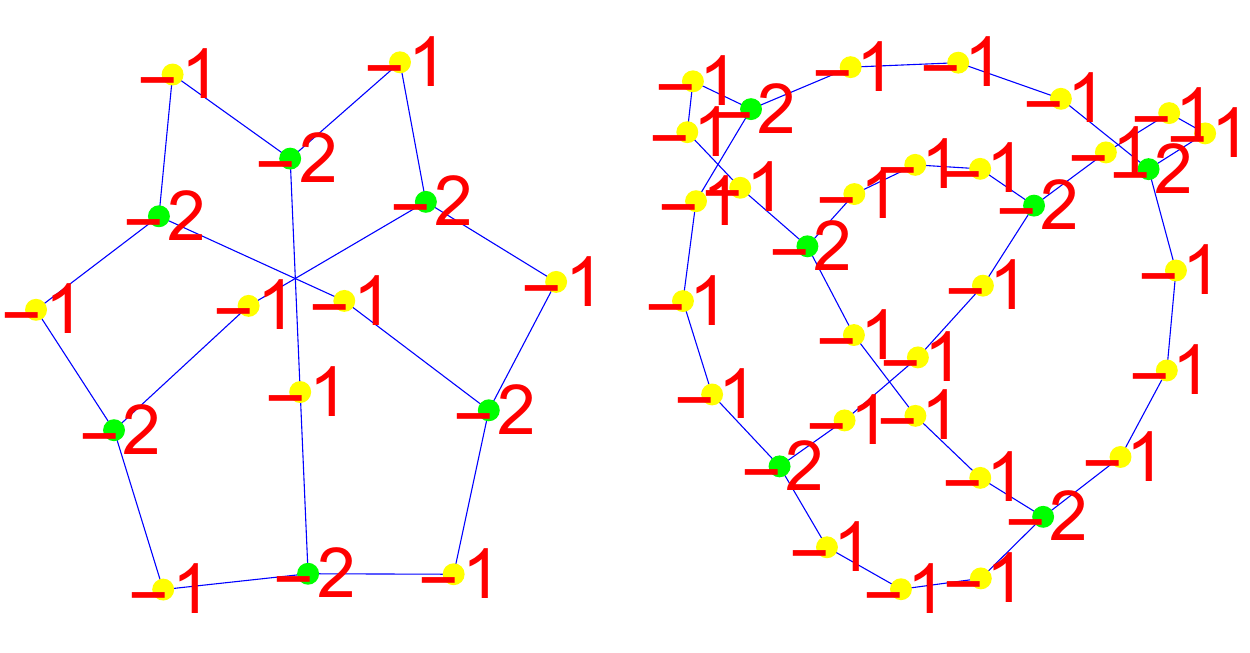}}
\caption{
Examples of graphs and their Barycentric refinement. 
In each case, we add the sphere spectral values at 
each vertex. 
\label{gallery}
}
\end{figure}

Let us call the number 
$$   i(x)=1-\chi(S(x)) $$ 
the Poincar\'e-Hopf index of the vertex $x$, where $S(x)$ is the unit sphere in the Barycentric refinement $G_1$
of the complex $G$. The name has been chosen because if $f$ is a function taking the 
maximum at $x$, then this is the usual Poincar\'e-Hopf index $i_f(x) = 1-\chi(S^-(x))$ \cite{poincarehopf},
where $S^-(x)$ is generated by the set of vertices $y$, where $f(y)<f(x)$. The name character 
is adequate since on the subgroup of complexes with $\chi(G) \neq 1$, 
it is a multiplicative character on a group constructed by the 
join operation on complexes. The set of possible values $\{ i(x) \; | \; x \in V(G) \}$ is
the Poincar\'e-Hopf spectrum of the graph. Since $i_{G + H}(x,y) = i_G(x) i_H(y)$, they form generalized
characters. The set of unit sphere topologies is the sphere topology spectrum. 
We can quantify this numerically for example by attaching to each 
vertex $x$ the Betti numbers $b_k(S(x))$ of the unit sphere.
The set $\{ b_k(S(x)) \; | \; x \in V(G_1) \}$ for example could be called the $k$'th 
Betti spectrum of the complex $G$. \\

The invariance of the sphere topology allows to check whether a graph 
is a Barycentric refinement of a complex.  If the unit sphere topology changes under Barycentric
refinement, then $G$ is not the Barycentric refinement of a complex. An example is the windmill graph
obtained by gluing three triangles along a common edge. While all unit spheres of $G$ are connected,
there is a vertex in $G_1$ with a disconnected unit sphere. 
We have seen before that if $\psi(G)$ is not in $\{-1,1\}$, then $G$ is not the connection
graph of a complex. An example is a triangle which has $\psi(G)=0$.

\section{Sphere topology}

In this section we look at discrete spheres and unit spheres $S(x)$ of Barycentric 
refinements of finite abstract simplicial complexes. Note that since we look at rather general
graphs, the unit spheres are rarely spheres but rather general graphs. \\

Barycentric refined graphs are natural as they are Eulerian \cite{knillgraphcoloring2}.
This means that they can be colored with a minimal number of colors, the color being the dimension
which the point has when it was a simplex of the complex. 
This is useful for graph chromatology \cite{knillgraphcoloring}: one approach to 4-color a planar
graph is is to embed it first into a 2-sphere, then fill out the interior 3-ball by cut it up using
edge refinements until it is Eulerian and so 4-colorable. Having colored the interior and not
cutting anything at the boundary, this colors the boundary with 4 colors. \\

The unit sphere $S(x)$ of a simplicial complex is a sub-complex consisting of all simplices in $G$ 
which are either contained in $x$ or which contain $x$. This is better seen in the Barycentric refinement
$G_1$, for which the simplices of $G$ are now vertices and the unit sphere is the geodesic unit sphere
with respect to the graph distance. \\

Given a simplex $x$, the set of subsets of $x$ is the complex $S^-(x)$, the stable or negative part of $S(x)$.
The set of simplices which contain $x$ define the sub-complex $S^+(x)$, the unstable or positive part of $S(x)$.

More intuitive is the graph theoretical reformulation in the graph $G_1$: the unit sphere $S(x)$ 
consists of all vertices in $G_1$ which are connected to $x$.  The function $f(x)={\rm dim}(x)$ 
which gives the dimension of $x$ when we look at it as a set in the simplicial complex $G$. 
It defines a coloring in $G_1$.  The dimension functional ${\rm dim}$ resembles a Morse function as 
it allows to partition the unit sphere $S(x)$ into a stable part $S^-(x)$ and unstable part 
$S^+(x)$. As in the continuum, the join of $S^-(x)$ and $S^+(x)$ is $S(x)$, this relation will also hold
here in full generality for any simplicial complex. The property $S^-(X) + S^+(x) = S(x)$ could serve
as a Morse condition without assuming $S^+(x)$ or $S(x)$ are spheres. But here, we don't make any such
assumption. \\

So here is again a formal definition for $S^{\pm}(x)$ in the case of the function $f(x)={\rm dim(x)}$: 
$$  S^-(x) = \{ y \in S(x) \; | \; f(y)<f(x) \} \;  $$
and
$$  S^+(x) = \{ y \in S(x) \; | \; f(y)>f(x) \} \; . $$

Despite the fact that the definitions for $S^-$ and $S^+$ look similar, there is a difference as we do not
have any kind of Poincar\'e duality in the case of a general complex. The stable sphere $S^-(x)$ is always 
a discrete sphere. It only deals with the inside of a simplex which after Barycentric refinement is of Euclidean structure.
The possibly crazy outside world, deals with connections of simplices which are not Euclidean even for very 
simple examples. For a one-dimensional star graph for example the central vertex $x$ is $0$-dimensional
and $S^+(x)$ consists of all edges containing $x$. \\

The Zykov join of two graphs $G,H$ is is the graph $G + H$ which has as vertex set 
the union of the vertices of $G$ and $H$ and where two vertices are either connected if 
they were already connected in $G$ or $H$ or then if they belong to different graphs. 
We simply call $G + H$ the join of $G$ and $H$. The join renders the category of graphs into
a monoid which can be augmented to an Abelian group on which functionals like 
$j(G)=-\psi(G)$ or $i(G)=1-\chi(G)$ are multiplicative. \\

A $d$-sphere or shortly $d$-sphere is a finite simple graph which has the property 
that every unit sphere is a $(d-1)$-sphere and such that removing a vertex from it renders the 
remaining graph contractible. This inductive definition gets started with the assumption that the empty 
graph $0$ (the zero element in the monoid) is the only $(-1)$-sphere. 
A $d$-graph is a finite simple graph for which all unit spheres are $(d-1)$ spheres.
A $d$-complex is a finite abstract simplicial complex for which its Barycentric refinement is a 
$d$-graph. \\

{\bf Remark}. The notion of Evako d-sphere and d-graph is equivalent to the notion of d-sphere within 
combinatorial d-manifolds in the theory of simplicial complexes. Notions as such have been put forward
in \cite{I94,Evako1994}. Related is the notion of a combinatorial d-manifold which is
a connected, pure (every facet has dimension d) finite abstract simplicial complex of maximal dimension $d$ 
for which each ridge (codimension-1-face) has two facets as a boundary and such that all vertex links are 
PL-homeomorphic to the boundary of a d-simplex. 
A combinatorial d-manifold is then a d-sphere if it is PL homeomorphic to the boundary of a d-simplex. 
These notions rely on Euclidean embddings. Forman's discrete Morse theory remains in the discrete:
a combinatorial d-sphere $G$ is PL-homeomorphic to a standard PL-sphere if and only if it is a Reeb sphere:
some Barycentric subdivision $G_n$ admits a discrete Morse function with exactly two critical points
\cite{forman98}.
Beeing a Reeb sphere is equivalent to being an Evako sphere (removing one vertex renders the complex 
contractible). Starting with the notion of Evako sphere has the advantage of avoiding Morse theory. 
In any case, detecting spheres can be difficult \cite{JoswigLutzTsuruga}.

We start with a lemma which is motivated by an analogue fact in the continuum, where a 
Morse function at a critical point $x$ defines two parts $S^{\pm}(x)$, the stable and unstable part
of the sphere. They are obtained by intersecting a 
small geodesic sphere $S(x)$ with the stable or unstable manifolds $W^{\pm}(x)$. 
The sphere $S(x)$ itself is then $S^- + S^+$. This is the picture for the gradient flow of a
Morse function $f$. In the discrete, in full generality, the dimension function $f(x) = {\rm dim}(x)$
can serve as a Morse function as we have a stable and unstable part at every point. As pointed
\cite{CountingAndCohomology}, the Morse cohomology in $G_1$ in this setting is directly equivalent to the 
simplicial cohomology of $G$. This gives hope for a more general Morse setting for functions with 
possibly much fewer critical points allowing to compute cohomology more efficiently than with
incidence matrices defined by the simplicial complex. 

\begin{lemma}[Unit sphere decomposition] 
If $G_1$ is the Barycentric refinement of a finite abstract simplicial complex $G$, then
for any vertex $x$ in $G_1$, the unit sphere $S(x)$ is the join of $S^-(x)$ and $S^+(x)$. 
\end{lemma} 
\begin{proof}
Both $S^{\pm}(x)$ are subgraphs of the unit sphere $S(x)$. 
The union of the vertex sets of $S^-(x)$ and $S^+(x)$ is the vertex set of $S(x)$. 
If $a \in S^-(x)$ and $b \in S^+(x)$, then $(a,b)$ is an edge in $S(x)$ because as simplices,
$a$ is a sub simplex of $x$ and $x$ is a sub simplex of $b$. 
\end{proof} 

We can rephrase that also by saying that the unit sphere of a vertex $x$ for which the simplex 
$x$ is not prime as a complex in $G$ (which is always the case if the dimension is positive) 
is not prime in $G_1$ but splits as $S(x) = S^-(x) + S^+(x)$.  \\

Lets now look at the structure of the stable unit sphere:

\begin{lemma}
The stable unit sphere $S^-(x)$ of a vertex $x$ in the 
Barycentric refinement $G_1$ of an abstract simplicial complex
is always a $d$-sphere for $d={\rm dim}(x)-1$. 
\end{lemma}
\begin{proof}
$S^-(x)$ is the boundary of the Barycentric refinement of the 
simplex $X$ in $G$ which belongs to the vertex $x$ in $G_1$. 
We can show by induction in the dimension $d$ that this is a sphere.
When looking at the intersection with a sphere $S(y)$ of a vertex $y$
in $S^-(x)$ the dimension is reduced by $1$, so that by induction, 
the unit sphere of a vertex $y \in S^-(x)$ is a $(d-1)$-sphere. Also
the second condition for a $d$-sphere, the collapsibility is satisfied
by induction. Take away a vertex $y$ of $S^-(x)$. Then every intersection
of $S^-(x)$ with an other sphere $S(z)$ which has contained $y$ becomes
contractible so that $S^-(x) \setminus y$ is contractible. 
\end{proof}

{\bf Examples}: \\
{\bf 1)} If $x$ belongs to a $0$-dimensional vertex $\{x_0\}$ in $G$, 
then $S^-(x)$ is empty so that it is a $d=(-1)$-dimensional sphere. \\
{\bf 2)} If $x$ belongs to a $1$-dimensional simplex $(x_0,x_1)$ in $G$,
then $S^-(x)$ consists of two isolated points $\{ x_0, x_1 \}$ which is 
$d=0$-dimensional sphere.  \\
{\bf 3)} If $x$ belongs to a two-dimensional simplex $(x_0,x_1,x_2)$ in $G$,
then $S^-(x)$ is a cyclic graph $C_6$ with $6$ elements:  \\
$\{x_0,x_0x_1,x_1,x_1 x_2,x_2,x_2 x_3,x_3,x_3 x_4,x_4,x_4 x_0 \}$.  \\

We will need a generalization of this which we can see as a higher 
order version of the just said. Given a vertex $x$ in the second
Barycentric refinement $G_2$, this corresponds to a simplex $(x_0, x_1, \dots, x_d)$
in $G_1$. We can assume without loss of generality that the vertices are
ordered so that $x_0 \subset x_1 \subset \cdots \subset x_k)$, where $x \subset y$
means that as a simplex in $G$, the simplex $x$ is contained in the simplex $y$. 
Now define the second stable sphere
$$ S^-(x) = \{ z \subset x \; | \; z \neq x_0, z \neq x_1 \dots, z \neq x_k \} \; . $$
These second stable spheres are of interest because they appear as unit spheres in 
the interior of the second Barycentric refinement $G_2$ of $G$. 

\begin{lemma}
The second stable unit sphere $S^-(x)$ of a vertex $x=(x_0, \dots, x_k)$ 
in the Barycentric refinement $G_2$ of an abstract simplicial complex $G$
is always a $(l-k)$-sphere if ${\rm dim}(x_k)=l$. 
\end{lemma}
\begin{proof}
We can see this in two different ways: 
first we can see it as an intersection of $k$ 
unit spheres $S(x_j)$ in the stable sphere $S^-(x_k)$. Each intersection
reduces the dimension by $1$. \\
If we mark the set of dimensions $k_j = {\rm dim}(x_j)$ in the 
integer interval $\{0, \dots, l \}$, then the set of simplices which 
belong to a gap of dimension values $[a,b]$ which are omitted, 
then this is a sphere of dimension $b-a-1$. We call this the gap sphere
belonging to $[a,b]$. The second stable unit sphere is now the join
of all these gap spheres:
$$ S^-(x) = S_{[a_1,b_1]} + \cdots + S_{[a_m,b_m]}  \; . $$
\end{proof}

{\bf Examples}: \\
{\bf 1)} The set of simplices in the complex of all 
complete subgraphs of $K_4$ which are between a vertex $p$ and the 
tetrahedron $x$ itself can be seen as the virtual stable sphere
$S^-(\{x,p\}) = S(x) \cap S(p)$ which is a stable sphere in the
Barycentric refinement. These are all the edges and triangles which 
contain $p$ are contained in $x$. This is a cyclic graph $S_6$. In 
this case, there is only one gap $[a,b] = [1,2]$. \\
{\bf 2)} The set of simplices in the complex of all
complete subgraphs of $K_5$ which are different from a fixed 
vertex $p$ and a fixed triangle $t$. In this case there are two 
gaps $[a_1,b_1] = [1,1]$ and $[a_2,b_2] = [3,3]$. The two 
gap spheres are $0$-dimensional spheres $P_2$ and their join is 
$C_4 = P_2 + P_2$. Again this can be seen as a stable sphere $S^-(\{x,p\})$
of a refinement, where the edge $\{x,p\}$ has become a vertex. 

\section{The sphere group}

In the continuum the monoid of spheres with join operation
is not that interesting as two spheres of the same
dimension are topologically equivalent. The map ${\rm dim}$ from 
$({\rm Topological spheres},+,0)$ to $\mathbb{N}$ is an isomorphism
of monoids. \\

This changes completely in a combinatorial setting. The sphere group is 
now interesting as there is a countable set of non-isomorophic spheres of each 
positive dimension. First of all we have to establish that the property of 
being a sphere  is invariant under addition given by the Zykov join operation.  \\

The join operation has the same properties as in the continuum:
it is associative, commutative and  preserves spheres. For associativity, note that
the graph $A  +  B  +  C$ has the union of the vertex sets as vertices. 
Two vertices $x,y$ are connected if they were connected in one of the three
components $A,B,C$ or then are in different components. 
The join of two simplices of dimension $n$ and $m$ is a simplex of dimension $n+m+1$: 
shortly $K_n+K_m=K_{n+m}$. The join of two spheres is again a sphere:

\begin{lemma}[Joins preserve Spheres]
The set $({\rm Spheres},+)$ is a monoid, a groupoid with a neutral element 
given by the $(-1)$-sphere $0$. The join of $n$-sphere and a $m$-sphere is a $n+m+1$-sphere.
\end{lemma}
\begin{proof}
Start that the join of two $(-1)$-spheres is again a $-1$ sphere.
Then use induction with respect to the sum $n+m$ of the dimensions.
We have to show that every unit sphere of $G + H$ is a sphere and
that removing a vertex of $G + H$ renders the remaining graph
contractible. Given a vertex $x$ in $G + H$. Assume it is in $G$.
Its unit sphere is the join of $S_G(x) + H$ which by induction assumption
is a sphere. The case when $x$ is in $H$, is analog.
If we remove the vertex $x$ in $G$, then by definition the graph $G \setminus x$
is contractible because the join of a contractible graph with any other unit
sphere graph is contractible.
\end{proof}

The Grothendieck construction produces then a group from this monoid. Lets call
it the sphere group. Elements in this groups are equivalence classes of
pairs $A-B$ of spheres, where $A-B \sim C-D$ if there exists a sphere $K$
such that $A+D+K = B+C+K$. The group element $O  - C_4$ with octahedron $O$ for example is $P_2$.
One can see this also because the octahedron is $O=3 P_2$ and the cyclic graph
$C_4$ is $2 P_2$ so that $3 P_2 - 2 P_2 = P_2$.  \\

A $d$-graph is a graph for which every unit sphere is a 
$(d-1)$-sphere. A $d$-complex is a finite abstract 
simplicial complex $G$ for which its Barycentric refinement 
$G_1$ is a $d$-graph. 

\begin{coro}[Geometric invariance]
The Barycentric refinement $G_1$ of a $d$-complex $G$ is a $d$-complex.
\end{coro}
\begin{proof}
Every unit sphere $S(x)$ is the join of two spheres $S^-(x)$ and $S^+(x)$. 
$S^-(x)$ is always a sphere. The fact that $S^(x)$ is a sphere follows in the
case of an original vertex by the assumption and for a simplex $(x_0, \dots, x_d)$
from the fact that $S^+(x)$ is the intersection of the spheres  $S(x_k)$ which by 
definition is a sphere. By the previous corollary, also $S(x)$ is a sphere. 
\end{proof}

While $S^-(x)$ is always a sphere as we have seen before, the unstable sphere $S^+(x)$
is not a sphere in general. For the central vertex $p$ in the star graph $G$ with 
$n$ spikes for example, the stable sphere $S^+(p)$ is $P_n$, the graph with $n$ vertices
and no edges, which is only a sphere if $n=2$.  In the case of a 
$d$-complex $G$, all unstable spheres $S^+(x)$ are spheres too. The unit sphere
decomposition $S(x) = S^-(x) + S^+(x)$ is then the addition of two spheres. \\

But in general, the structure of the stable sphere $S^+(x)$ also can be described by 
unit spheres. 

\begin{lemma}[The positive sphere lemma]  
If $G_1$ is the Barycentric refinement of a simplicial complex $G$ and if
$x=(x_0 \dots x_k)$ with $G$-vertices $x_i$ is a vertex in $G_1$, then 
$S^+(x) = \bigcap_{j=0}^k S(x_j)$.           
\end{lemma}
\begin{proof}
We only have to show $S^+(x) = \bigcap_{j} S(x_j)$. Given a vertex $y \in S^+(x)$, then
$y = x_0 \dots x_k y_1 \dots y_l$. But this is in the intersection.
On other other hand, if we are in the intersection, then we have to be a super simplex
of $x_0,\dots,x_k$. In other words, we have to be in $S^+(x)$.
\end{proof}

{\bf Example:} Let $G$ be the octahedron and $G_1$ its Barycentric refinement.
Now, if $x=x_0$ is a vertex in $G_1$ which was an original vertex, then $S(x) = 0 + S(x)$
where $0$ is the empty graph. If $x=(x_0 x_1)$ belonged to an edge, then 
$S^-(x)$ is a $0$-sphere and $S^+(x)$ is a $0$-sphere, as an intersection of two spheres. 
Now $S^+(x) = S(x_0) \cap S(x_1)$. \\

Given a vertex $x$ in $G_2$ we define the second positive sphere $S^+(x)$ 
as part of the unit sphere $S(x)$ consisting of vertices $z$ in $G_2$ 
which when written as a simplex $Z$ in $G_1$ contains the simplex $X$ corresponding to $x$. 

\begin{lemma}[Second positive sphere]
Given a vertex $x=(x_0, \dots, x_k)$ in $G_2$ with $x_0 \leq \dots \leq x_k)$.
The second positive sphere $S^+(x)$ in $G_2$ is graph isomorphic to the
Barymetric refinement of $S^+(y)$ in $G_1$ with $y=x_k$.
\end{lemma}

\begin{proof}
Given $z \in S^+(x)$, then $x_k \leq z$ and $z=(x_0,\dots_k,x_{k+1}, \dots x_{l})$.
This point $z$ is in $S^+(y) = \bigcap_j S(x_j)$.\\
On the other hand, given $w \in S^+(y)$, then this is a simplex containing $y$
so that $(x_0, \dots, x_{k},w)$ this defines a new vertex $z$ in $G_2$ which 
is connected to $y$. It belongs therefore to $S(x)$. 
\end{proof}

Let $x$ be a vertex in $G_2$ with $x=(x_0, \dots, x_k)$ in $G_1$ with 
${\rm dim}(x)=k$ so that $S(x)$ is the join of a $(k-1)$-sphere $S^-(x)$ and $S^+(x)$. 
We have defined $S^+(x)$ and $S^-(x)$ for vertices $x$. 
Assume the vertex $x$ in $G_2$ belongs to the simplex $x =(x_0 \dots x_k)$ in $G_1$. 
Define $S^-_1(x)$ as the simplex $(x_0 \dots x_k)$ equipped with the $(k-1)$-skeleton complex structure, 
so that it is a sphere. Define also the virtual positive sphere $S^+(x)=\bigcap_j S(\{ x_j \})$ in $G_1$. 
The graph $S^-_1(\overline{x}) + S_1^+(\overline{x})$ 
is the virtual unit sphere of the simplex $x$ in $G_1$. It is topologically equivalent to $S(x)$
in $G_2$. We will see in a moment that it is equivalent to an actual unit sphere $S(y)$, where $y=x_k$.  \\

First a simple property of Barycentric refinement: 

\begin{lemma}[Intersection complexes]
Let $H,K$ be subcomplexes of G so that $H_1,K_1$ are subgraphs of $G_1$.
Then $(H \cap K)_1  = H_1 \cap K_1$.
\end{lemma}

\begin{proof}
Given a simplex $x$ in $H \cap K$. This means $x=\{x_0,...,x_k\}$ is a set 
which is both in $H$ and $K$ and all are connected. Now, also $(x_0, ...,x_k)$ is
a simplex in $H_1$ as well as in $K_1$. The simplex is in the intersection.        
on the other hand, given a vertex $x$ in $H_1 \cap K_1$. Then $x$ is both in $H_1$
and $K_1$, meaning that $x=\{ x_0, \dots, x_k \}$ is a face in $H$ and in $K$.
\end{proof}

For every vertex $x$ in $G_2$ corresponding to a vertex $(x_1, \dots, x_k) \in G_1$, 
each of the elements $x_j$ are vertices in $G_1$. 
The following lemma makes a step to relate $(x)$ with a new unit sphere $S(y)$. 

\begin{lemma} 
$S^+(x)$ with $G_2$ is the Barycentric refinement of a virtual positive sphere 
in $G_1$ which agrees with the positive sphere $S^+(y)$ in $G_1$. 
\end{lemma} 
\begin{proof}
$(x_0,x_1, \dots ,x_k) \subset G_1$. The positive sphere $S^+_2(x)$ is
the Barycentric refinement of $\bigcap S(x_k)$ in $G_1$. Each $x_k$
corresponds to a simplex $X_k$ in $G$.  We have now to find a vertex $y \in G_1$ 
such that $\bigcap S(x_j)$ is equivalent to $S^+(y)$. Such an $y$ corresponds to
a simplex $(y_0,y_1, \dots, y_k)$ in $G$. We can construct that $k$-simplex
from the simplices $X_0, \dots, X_k$ inside $G$. The simplices $X_k$ are 
all contained in each other since the $x_j$ are all connected. 
Assume without loss of generality that things are numbered so that
$X_0 \subset X_1  \cdots \subset X_k$.  We can now pick 
$y_0 \in X_0, y_1 \in X_1 \setminus X_0, y_2 \in X_2 \setminus X_1, \dots X_k \setminus X_{k-1}$.
Now $S^+(y)$ is a subgraph of $G_1$ containing all vertices which belong to 
simplices in $G$ containing all the $y_j$. $S^+(y)$ is $\bigcap S(\{y_j\})$ 
which is the Barycentric refinement of $\bigcap_j S(y_j)$.
$S^+(x)$ is $\bigcap S(x_j)$ whose Barycentric refinement is $S^+(x)$ in $G_2$. 
\end{proof} 

{\bf Examples.}\\
{\bf 1)} If $x=(x_0)$ is $0$-dimensional and $(x_0)$ is the vertex in $G_1$ which belongs to $x$,
then we can take $y=x_0$. The sphere $S(y)=S^+(y)$ has as a Barycentric refinement
the sphere $S(x) = S^+(x)$.  \\
{\bf 2)} If ${\rm dim}(x) = 1$ and $x=(x_0,x_1)$ then we get the virtual unit sphere 
$S( (x_0,x_1) )  = P_2 + S^+(x)$ in $G_1$. The graph $S^+(x)$ in $G_2$ is isomorphic
to the refinement of $S(x_0) \cap S(x_1)$. 

\section{The twin propositions} 

Given a connection graph $G'$ of a complex $G$, we can either remove a vertex $x$ in $V(G')=V(G_1)$
or we can add a vertex $x$. We have seen already that adding a vertex changes the
Fredholm determinant by a factor $i(x)$. It is a bit surprising but in both cases, 
the Fredholm characteristic $\psi(G) = {\rm det}(1+A'(G))$ gets multiplied by an integer. \\

Here is again the case when adding a cell. It is a bit more general than before as we now do 
not insist that the new cell has a unit sphere smaller dimensional cells only. The proof is the
same as before and uses the concept of a valuation on simplicial complexes.  This is a
functional $X$ on subcomplexes, which satisfies the linearity condition
$X(H \cup K) + X(H \cap K) = X(K) + X(H)$ (we can not write $X(H \Delta K) = X(K) + X(H)$ as
$H \Delta K=H + K$ is not a complex any more but it explains the name linearity as 
the symmetric difference $\Delta$ is the addition in a larger Boolean ring of structures). The
prototype of a valuation is the Euler characteristic. As in the continuum, where valuations are
defined on convex subsets of Euclidean space (which naturally come with a simplicial complex structure). 
It is important to see that a valuation is not a functional on subsets like in measure theory and indeed
valuations do not extend to the measure theory level except in geometric measure theory, but where we
deal with a Fermionic analogue of valuations. \\

In the discrete also, it is not just a functional on subsets. The Euler characteristic for example
depends on the choice of simplicial complex structure. The Euler characteristic is the prototype of
a valuation. Actually there is a $({\rm dim}(G)+1)$-dimensional space of valuations by discrete Hadwiger
and the Euler characteristic is up to a scaling factor the only 
which is a combinatorial invariant: the Barycentric operator on $f$-vectors has
only one eigenvalue $1$ and eigenvector $(1,-1,1,\dots)$. Repeating the Barycentric refinement leads to 
a central theorem \cite{KnillBarycentric2}. 
In the proof, we again make use of the fact that the only valuation which assumes the value $1$ 
on all simplices must be the Euler characteristic. 

\begin{propo}[Attaching cells]
Let $G$ be a simplicial complex and $G'$ its connection graph. 
If we chose a sub-complex $H$ of $G$ leading to a subgraph $H'$ of $G'$ and 
attach a new cell $x$ to it, leading to a new complex $G'\cup_H x$, then the Fredholm 
characteristic $\psi$ of $G+x$ gets multiplied by a factor $i(x)=1-\chi(H)$:
$$ \psi(G \cup_H x) = (1-\chi(H)) \psi(G) \; . $$
\end{propo}
\begin{proof} 
(i) The map $H \to X(H) = \psi(G \cup_H x)-\psi(G)$ is a valuation. 
The reason is geometric: as this is a super count of the number of new 
paths which are allowed when adding the vertex $x$. Since $x$ is a single 
vertex and only one path can occupy it, the linearity condition is satisfied. \\
(ii) Also the map $H \to Y(H) = X(H)/(-\psi(G))$ is an integer valued valuation.
The reason is that by unimodularity (or induction assumption if
we want to use this proposition as a proof), $\psi(G)$ is $\{ -1, 1\}$ valued
and because when looking at a function of $H$, the complex $G$ is fixed. \\
(iii) The map $H \to Y(H)$ takes the value $1$ on simplices $H$. The reason is
that $H \cup \{x\}$ is again a simplex and that we look now at the collection of
all Leibniz paths in $G' \cup_H \{x\}$ which go through $x$. Now $Y(H)=1$
is equivalent to $X(H)=-\psi(G)$ and so $\psi(G \cup_H x)=0$. The reason why this
is zero is because $\psi(K_n)$ is zero for every complete graph $K_n$ with $n>1$.
every path in $G \cup_H x$ can now be paired with an other path of different parity.\\
(iv) Because all valuations which take the value $1$ on simplices must be the
Euler characteristic, we have $Y(H) = \chi(H)$. 
\end{proof}

In the new CW complex, $S(x)$ is the unit sphere in the Barycentric 
refinement graph of $G+x$. \\

Now lets start with a complex and write $G-x$ for the complex, where the cell $x$ has been removed.
Now what matters is the sphere $S(x)$ in $G_1$ and not the sphere in $G'$. 
The reason is that if we remove the cell x, we remove all links to sub-simplices and super simplices. 
The proof is very similar. We write it again down so that we see the difference.

\begin{propo}[Removing cells]
If $G$ is a simplicial complex and $G'$ its connection graph. 
If we chose a vertex $x$ in $G'$ and remove $x$, leading to the new complex 
$G'-x = (G-x)'$, then the Fredholm characteristic of $G-x$ 
is the Fredholm characteristic of $G$ multiplied by a factor $i(x)=1-\chi(S(x))$.
$$ \psi(G-x) = (1-\chi(S(x))) \psi(G) \; . $$
\end{propo}
\begin{proof}
(i) The map $H \to X(H) = \psi(G) - \psi(G - x)$ is a valuation.
(ii) Scale it to $Y(H) = X(H)/\psi(G)$. Again this is integer valued
complete subgraphs. \\
(iii) The valuation $Y$ takes the value $1$ on complete subgraphs. \\
(iv) As any valuation taking the value $1$ on complete subgraphs is Euler characteristic, we 
 get $Y = \chi(H)$. This implies $\chi(S(x)) = X(S(x))/\psi(G)
= (\psi(G) - \psi(G - x))/\psi(G) = 1-\psi(G - x)/\psi(G)$ and so
$\chi(S(x))-1 = -\psi(G-x)/\psi(G)$ which is equivalent to the claim. 
\end{proof}

\section{Unit sphere topology}

We establish now that the set of unit sphere topologies
is stable from $G_1$ on. Further Barycentric refinements do not produce new
sphere topologies:

\begin{thm}
a) Every unit sphere $S(y)$ in $G_1$ is topologically equivalent to a sphere $S(x)$ in $G_2$.  \\
b) Every unit sphere $S(x)$ in $G_2$ is topologically equivalent to a sphere $S(y)$ in $G_1$. 
\end{thm}
\begin{proof}
a) The Barycentric refinement $S(x)_1$ is isomorphic to $S((x))$ in $G_2$, if
$(x)$ is the vertex in $G_2$ which corresponds to the singleton vertex $(x)$. \\
b) Given now a unit sphere $S(x)$ in $G_2$. It can be decomposed as 
$S^-(x)  +  S^+(x)$. If $x=(x_0, \dots, x_k)$ is the simplex in $G_1$ which belongs
to the vertex $x$ in $G_2$, take $y=x_k$. We have seen that
$S^-(y)$ is a sphere of the same dimension than $S^-(x)$. It is therefore topologically 
equivalent. Furthermore, we have seen that $S^+(x)$ is the Barycentric refinement of $S^-(x)$.
It is also topologically equivalent. 
\end{proof}

{\bf Examples:} \\
{\bf 1)} If $G$ is a one dimensional graph, then the set of unit spheres of $G$ 
and the set of unit spheres of $G_1$ agree as we have $S(x)=P_{d(x)}$ with $d(x)$
meaning the vertex degree. But this is special as the Barycentric refinement of a 
graph without edges is the graph itself.  \\
{\bf 2)} If $G$ is a graph with only one dimensional spheres, then this is true too 
for $G_1$. For an octahedron or icosahedron from example, the unit spheres are cyclic
graphs. After parametric refinement, this remains so.  \\

The next lemma assures that the join operation honors topological equivalence. 
Now this is obvious if we look at topological realizations in some Euclidean space. 
Since we don't want to use the functor of topological realizations of abstract simplicial 
complexes to actual simplicial complexes, and since we don't need topological equivalence
in the sense of topology \cite{KnillTopology}, lets for the following just take a weaker
notion of equivalence. We only look at the Euler characteristic
because this is the only thing we really need for establishing the Barycentric
invariance of the Green function diagonal values. \\ 

So, we say that two finite abstract simplicial complexes are topologically equivalent 
if their Euler characteristic agree: 

\begin{lemma}
Given two pairs of simplicial complexes $A,B$  and $C,D$.
If $\chi(A) = \chi(C)$ and $\chi(B)=\chi(D)$, then
$\chi(A+C)=\chi(B+D)$. 
\end{lemma}
\begin{proof}
This follows from fact that $G \to 1-\chi(G)$ is a group homomorphism. 
$1-\chi(A+C) = (1-\chi(A)) (1-\chi(C))$. 
\end{proof}

Note however that the join does not commute with the Barycentric refinement:  \\

{\bf Examples.} \\
{\bf 1)} The Barycentric refinement $(C_4+C_4)_1$ of the join $C_4 + C_4$ is a graph
with $80$ vertices. It is larger than the join of the refinements 
$C_8 + C_8 = (C_4)_1 + (C_4)_1$ which is a graph with 16 vertices. 
Still, both are 3-spheres.  \\
{\bf 2)} The Barycentric refinement $(P_3+P_3)_1$ of the utility graph $P_3+P_3$
is a graph with 15 vertices and 18 edges. The join of the Barycentric refinements
is the utility graph itself: $(P_3)_1 + (P_3)_1 = P_3 + P_3$.  \\

{\bf Remarks:} \\
{\bf 1)} More generally if $H \sim H'$ and $K \sim K'$, then $H \sim K$ is equivalent to $H' \sim K'$.
This could be shown most conveniently using an embedding. \\
{\bf 2)} It is not true that $H_1 +  K_1 = (H + K)_1$ as the example of $H_1=C_4,K_1=P_2$ shows. 
In this case both are equivalent 2-spheres. 

\section{Prime and prime connection graphs} 

Given an integer $n \geq 2$, we can look at two graphs $G_n$ and $H_n$. Both
have the vertex set set $\{2, \cdots, n \}$. In the {\bf prime graph} 
$G_n$ case, vertices are connected, if one divides the other. In the 
prime connection graph $H_n$, two vertices are connected, if they have a 
common divisor larger than $1$. 
The graph $G_n$ is part of the Barycentric refinement of the complete
graph on the set of primes and the graph $H_n$ is part of the connection
graph of the complete graph on the set of primes.  \\

We can now express the sphere spectrum of $G_n$ in terms of the adjacency 
matrix $A'$ of $H_n$. It is an application of the connection we have 
built between the Green function values, the diagonal values of the 
inverse of the Fredholm matrix, and the Euler characteristic of the 
corresponding sphere. 

\begin{coro}
$1-\chi_{G_n}(S(x)) = (1+A_{H_n})^{-1}_{xx}$. 
\end{coro}
\begin{proof}
Given $n$, let $P=\{p_1, \dots, p_k \}$ be the primes which appear
as factors of square free integers in $V_n=\{2, \dots, n\}$. 
This defines now a abstract finite simplicial complex $G$. The sets
are the subsets of $G_n$ such that their product is in $V_n$. 
Because for any given $A$ and any non-empty subset $B$ of $A$
also the product of the elements in $B$ is in $V_n$, this is indeed
an abstract finite simplicial complex. The Barycentric refinement $G_1$
of $G$ is now the graph $G_n$ (there should be no confusion since $n >1$
prevents the Barycentric refinement $G_1$ to be mixed up with $G_n)$.
The connection graph $G'$ is the graph $H_n$. 
\end{proof}

\begin{figure}[!htpb]
\scalebox{0.4}{\includegraphics{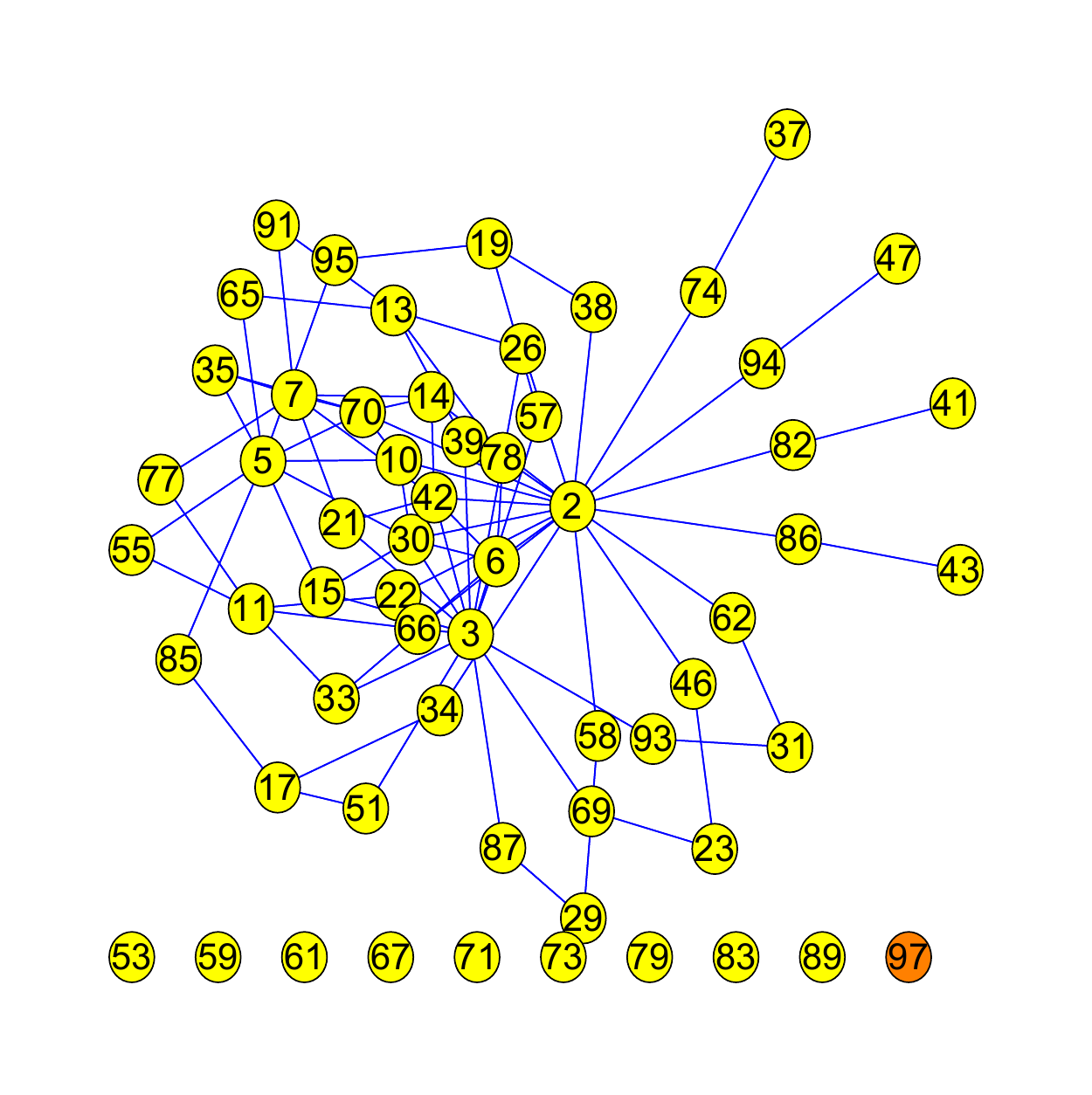}}
\scalebox{0.4}{\includegraphics{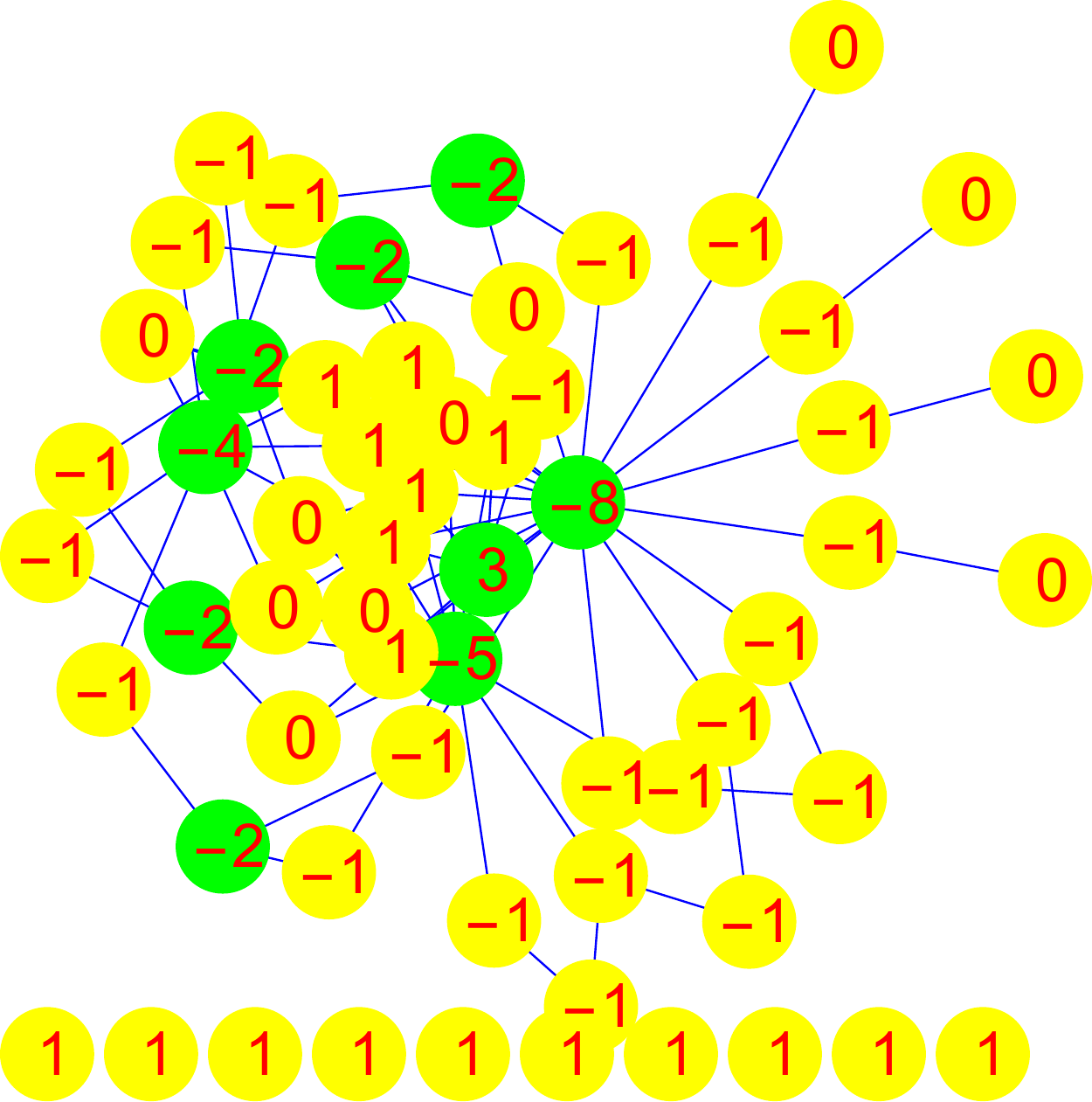}}
\caption{
We see the prime graph $G_{100}$, first with the vertex labels given by 
square free integers. In the picture to the right we see the 
sphere spectrum $i(x) = 1-\chi(S(x))$. For primes $x$, we have $i(x)=1$. 
The unit sphere of $x=2$ is a graph 
with vertex set $\{ 6, 10, 14, 22, 26, 30,$ $34, 38, 42, 46, 58,$ 
$62, 66, 70, 74,$  $78, 82, 86, 94 \}$, 
which has $v_0=10$ connectivity components. There is furthermore
a nontrivial loop $6,42,14,70,10,30,6$ in the unit sphere
which gives a non-trivial fundamental group and $b_1=1$. There are 
no triangles in $S(x)$ so that $\chi(S(x))=1-(10-1)=-8$. 
\label{primegraph}
}
\end{figure}

We can wonder now how to estimate the Green function values from above
or from below for the graphs $G_n$.

\section{Graph arithmetic}

The multiplicative monoid
$(\mathbb{N}, \cdot) = (\{ 1,2,3, \dots, \}, \cdot)$ has the unique
factorization property, telling that every non-unit can be factored 
uniquely into a product of prime numbers. It is a unique factorization monoid.
More generally, the multiplicative group of any 
unique factorization domain is a unique factorization monoid. 
The additive monoid of $\{ 0,1,2 3, \dots, \}$ also has unique
prime factorization but it is trivial since only $1$ is an additive
prime and any number $n$ naturally can be decomposed as $n=1+1+1+\dots +1$. 
The question of prime factorization can be extended to networks. 
It turns out that there is both an addition and multiplication which are
compatible by distributivity. We work here with graphs even so one can 
take more general simplicial complexes. \\

The set $\mathbb{G}$ of finite simple graphs $G=(V,E)$ becomes 
an additive monoid with the join operation $G  +  H$. The definition
is as follows: we have $V(G  +  H) = V(G) \cup V(H)$ and 
$E(G  +  H) = E(G) \cup E(H) \cup \{ (a,b) \; | \; a \in V(G), b \in V(H) \}$. 
The $0$ element is the empty graph $(\emptyset,\emptyset)$. Adding one to $G$ 
is the cone over $G$ and adding the $0$-sphere $P_2$ to $G$ is the 
suspension of $G$. \\

The maximal dimension ${\rm \overline{{\rm dim}}}(G)$ of a graph $G$ 
is the dimension of the largest complete subgraph $K_{d+1}$
which appears in $G$. The number ${\rm \overline{{\rm dim}}}(G) + 1$ 
is also called the clique number. The relation $\overline{{\rm dim}}(H  +  K)
= \overline{{\rm dim}}(H) + \overline{{\rm dim}}(H) + 1$ follows from 
$K_n + K_m = K_{n+m}$.  \\
(Less obvious is the same inequality for the inductive dimension ${\rm dim}(G)=
\sum_{v \in V(G)} [1+{\rm dim}(S(x))]/|V|$):
$$ {\rm dim}(G + H) \geq {\rm dim}(G) + {\rm dim}(H) + 1 \;. $$ 

Lets call a graph $G$ an additive prime graph if it can not be written as
$G=H  +  K$ with $H,K$ different from the $0$ element. An example of 
an additive prime graph is $K_1$. All the $K_n$ can be factored as 
$K_n = K_1  +  K_1  +  K_1 
= K_1^n$ and are therefore not prime. A star graph $S_n$ can be factored as $P_n  +  K_1$
where $P_n$ is the graph with $n$ vertices and no edges. 

\begin{lemma}
Every finite simple graph $G$ can be decomposed into additive 
prime factors $G = p_1 + \cdots + p_n$, where $p_i$ are prime graphs. 
\end{lemma}

\begin{proof}
If there is a factor $H$, then its maximal dimension has to be smaller. 
By induction, we can then break up every factor into smaller factors or leave it
if it if it is an additive prime. 
\end{proof} 

\begin{lemma}
There are many additive prime graphs: every disconnected graph is an additive
prime, every non-contractible graph with a prime volume is an additive prime. 
\end{lemma}

\begin{proof}
The join of two graphs is connected so that a disconnected graph must be prime.
Every circular graph $C_n$ with $n>4$ or every point graph $P_k$ is prime. 
For the circular graph, one can see that if $C_n = K  +  H$, 
then both $K,H$ have to have maximal dimension $0$. But we can list $P_1  +  P_1=K_2$,
$P_1  +  P_2 = L_3$, $P_2  +   P_2 = C_4$. Any other $P_k  +  P_l$ contains
a star graph $S_k$ or $S_l$. If both $k,l>2$, then this is incompatible with $C_n$. 
\end{proof}

As we can list all one-dimensional graphs which are not prime we see that most
one dimensional graphs are prime. For example, every tree which is not a star graph
is prime. \\

As in the case $\mathbb{N}$, the question about unique prime factorization is more
difficult than the question of the existence of a factorization. 
Experiments indicate that graphs form a unique factorization monoid: still the question
remains: 
is it true that every graph has a unique factorization $G=p_1 + \cdot + p_n$ 
into prime graphs $p_k$. \\

Note that this is not true for the join in the standard topology 
case as if we take the join of 
a $k$-sphere and a $m$-sphere, we get a $k+m+1$ sphere. So, we can for example
take a $3$-sphere and add a $4$-sphere, which gives topologically
the same than adding a $2$-sphere to a $5$ sphere. In each case, we get 
a $8$-sphere. In the discrete $C_4=P_2 + P_2$ but $C_5$ can not be written as a sum.  \\

Here is a formula which gives the $f$-vector for the join of two complexes. 
Given a simplicial complex with $v_k$ simplices of dimension $k$. 
Define the extended $f$-vector generating function  = extended Euler polynomial of $G$ as  
$$   f_G(x)= 1+\sum_{k=0} v_k x^{k+1} \; . $$
The empty graph has the generating function $f_G(x)=1$.
(Originally, I had defined $f_G(x)=1/x+\sum v_k x^k$ and defining the multiplication as
$x f_G(x) g_G(x)$ but June-Hou Fung pointed out to me that a multiplication with $x$ 
simplifies the generating function.)

\begin{propo} The join $G + H$ of two complexes has the
$f$-generating function of the form 
$$   f_{G + H} = f_G f_H \; . $$
\end{propo}
\begin{proof}
Given a simplex $x$ in $G$ and a simplex $y$ in $H$, this
produces a simplex $x + y$ of dimension ${\rm dim}(x)+{\rm dim}(y) +1$. 
\end{proof} 

This formula can shed light on the unique prime factorization problem.
If $G+H$ can be factored, then its rational function can be factored in the 
above sense. We have then for example
$$  f_{p_1+p_2+p_3} = f_{p_1} f_{p_2} f_{p_3} \; . $$
But even if we have a unique factorization algebraically in the polynomials
with multiplication $f g(x) = f(x) g(x)$, we still have the problem to see
whether the factors can be realized as simplicial complexes and furthermore
require that the factors have positive entries. \\

Since $\chi(G) = 1+f(-1)$. We have $\chi(G + H ) = 1+(-1) f_G(-1) f_H(-1)
= 1-(\chi(G)-1)(\chi(H)-1) = \chi(G) + \chi(H) - \chi(G) \chi(H)$: \\

\begin{coro} 
The Euler characteristic of the join of two complexes satisfies 
$$  \chi(G+H) = \chi(G) + \chi(H) - \chi(G) \chi(H)  \; . $$
\end{coro} 

Of especial interest is the Euler-Poincar\'e functional $i(G) = -f(-1) = 1-\chi(G)$. 
(Again, it was June-Hou Fung who pointed out to me that the above corollary implies that $1-\chi(G)$ is
multiplicative). From the above formula we see that it behaves like a multiplicative character on the group 
of graphs. However, since we have to define $i(-G)=1/i(G)$ to extend it to the entire group,
we can not define $i(-G)$ for graphs $G$ which have $i(G)=0$. We see especially that $i$ is not
defined on the ``integers" $\mathbb{Z}$ given by the graphs $K_n$ and their negatives $-K_n$ as $i(K_n)=0$. 

\begin{coro}
For any two simplicial complexes $G,H$, the functional $i(G)=1-\chi(G)$ satissfies
$$  i(G+H) = i(G) i(H) \; . $$
\end{coro} 

While $i(G)$ can be zero, we can look at the subgroup of the join group for which $i(G) \neq 0$. 
This is still a group. 

\begin{coro} 
The monoid of graphs with join $+$ has the following sub-monoids: \\
a) the submonoid of graphs with even $\chi(G)$  \\
b) the submonoid of graphs with odd $\chi(G)$  \\
c) the submonoid of graphs with zero $\chi(G)$ \\
d) the submonoid of graphs with $\chi(G) \in \{0,2\}$ \\
e) the submonoid of sphere graphs \\
f) the submonoid of graphs with $i(G) \neq 0$. 
\end{coro}
\begin{proof}
This all follows from the sum formulas for $\chi$ and $i$ as 
well as the fact that the join leaves spheres invariant. 
\end{proof}

\begin{coro} 
$i(G)$ can be extended to a multiplicative character on the additive 
group of spheres. 
\end{coro} 
\begin{proof}
It takes there the values $\pm 1$ on the monoid. Define $i(-G)=i(G)$. 
Now $i(G-H) = i(G)/i(H)$. The functional $i$ takes values in $\{-1,1\}$.  
\end{proof}

{\bf Examples.} \\
{\bf 1)} If $F=P_n,G=P_m$ are both $0$-dimensional point graphs, then 
$f(P_n) f(P_m) = (1+nx) (1+mx) = 1+nx+mx + nmx^2$ showing that
the graph $F + G$ has $nm$ edges and $n+m$ vertices and Euler characteristic
$n+m-nm$. Looking at the number of edges, we see immediately
that a $1$-dimensional graph with a prime number of edges must be prime.
More generally, any graph with prime volume must be prime. \\
{\bf 2)} If $G + H$ is $2$-dimensional, then one of the graphs, say $G$ has to be 
$0$-dimensional and so $f_G(x) = 1+nx$. If $f_H(x) = 1+ mx+kx^2$ is the extended
Euler polynomial of a $1$-dimensional graph. There are $f=kn$ triangles and
$e=k+mn$ edges and $v=m+n$ vertices.
We see also the Euler formula confirmed as $\chi(G+H) = \chi(G) + \chi(H) - \chi(G) \chi(H) 
= n+(m-k) - n(m-k)$ which agrees with $v-e+f$. \\
{\bf 3)} If $G,H$ are two spheres, where at least one has Euler characteristic $2$, 
then their product has Euler characteristic $2$. Only if both have Euler characteristic $0$, it is
possible that the product has Euler characteristic $2$. We can define the monoid of all graphs
with Euler characteristic in $\{0,2\}$. We could also look at the sub-monoid of all graphs
with even Euler characteristic, or the monoid with odd Euler characteristic, or the set
of graphs with $0$ Euler characteristic. These graphs do not change the Euler characteristic of the 
graph it is multiplied with. \\
{\bf 4)} Lets look at two $1$-spheres $C_n$ and $C_m$. Their product is a 3-sphere with 
$f_{C_n + C_m}(x) =  (1+nx+nx^2) (1+mx+mx^2)$. The new $f$-vector 
therefore is $(m+n,m+n+mn,2mn,mn)$ which has zero Euler characteristic 
as $\sum_k (-1)^k v_k=0$. \\

A special case of the unique prime factorization problem in the sphere monoid is the question 
whether any $3$-sphere is either prime or then a unique sphere of two $1$-spheres. In this
simple case, we can answer unique factorization: 
it boils down to the question whether $a+b=m+n, a b = m n$ implies $m=a,n=b$
or $m=b,n=a$ which is true. But this could be a case for the Guy law of small numbers. 
We could imagine for example that there would be a $7$-sphere which can be written
as $p_1+p_2+p_3$ with prime $2$-spheres $p_i$ but also be able to write it as 
$q_1+q_2$ as a sum of two prime $3$-spheres. We have just not seen such an example yet
and it probably does not exist. \\

As class field theory in rings of integers like $\mathbb{Z}[\sqrt{-5}]$ shows, 
many monoid belonging to rings of integers in function fields do not have a 
unique prime factorization and this could be the case here too. \\

{\bf Remarks.} \\
{\bf 1)} The join $G + H$ can be obtained within an extended Stanley-Reisner ring as follows.
If $V=\{ x_1, \dots, x_n \}$ and $W = \{ y_1, \dots , y_n \}$ represent the vertices
of the graph $G$ and $H$ and $f = 1 + \sum_I a_I x^I, g=1+\sum_J b_J x^J$ represent the 
complex for $G$ and $H$, then $f g$ represents the complex for $G + H$. For example,
if $f = 1+x+y+z+xy+yz+xz+xyz$ represents the Whitney complex of $G=K_3$ and 
$g = 1+a+b+ab$ represents the complex for $H=K_2$, then 
$fg = 1+x+y+z+a+b+xa+xb+ya+yb+za+zb+xya+xyb+xza+xzb+yza+yzb + xyza$ represents the complex $G + H$. \\
{\bf 2)} Trying to generalize a modular arithmetic, one could try to modify the join operation to
subgraphs of a given background graph $P$ and define the graph with $V(G  +  H) = V(G) \Delta V(H)$
and $E(G  +  H)$ containing all edges contained in $V(G) \Delta V(H)$ or then 
connecting two different vertices $a \in V(G) \cap V(G) \Delta V(H)$ 
with $b \in V(H) \cap V(G) \Delta V(H)$. Mathematically we would look at a Stanley-Reisner ring
with $Z_2$ valued coefficients over a fixed number of variables. We have not succeeded to get
finite rings like that. \\
{\bf 3)} Any bipartite graph is a subgraph of $P_n + P_m$ as by definition, we can split the vertices into
two disjoint subsets. The graph $P_n + P_m$ is the most extreme case. The utility graph $P_3 + P_3$ 
is the most famous example. \\

\begin{figure}[!htpb]
\scalebox{0.7}{\includegraphics{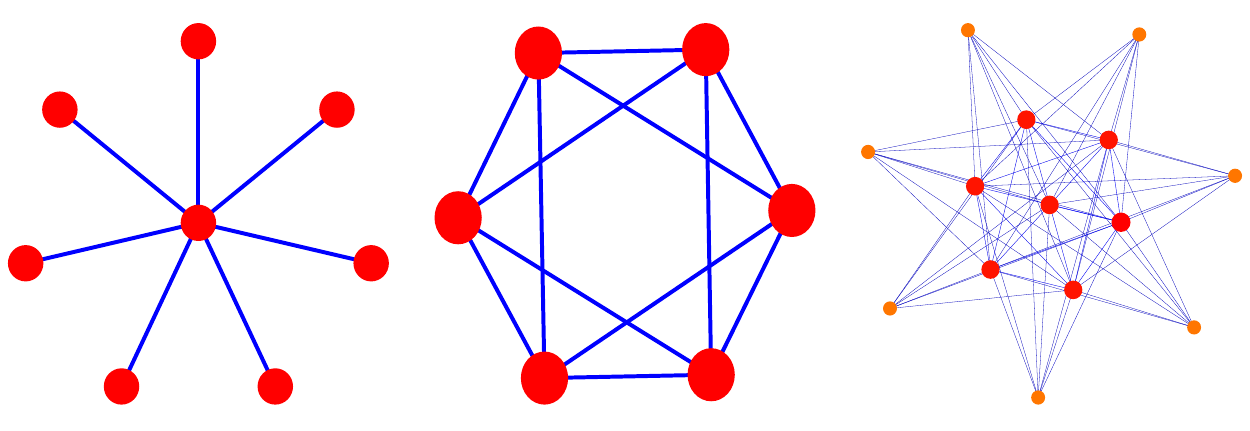}}
\caption{
We see the addition of a star graph $S_7$ and an octahedron
graph. Since the star graph is a cone $S_7 = P_7  +  K_1 = P_7+1$
and the Octahedron graph factors into ${\rm Oct} = K_2  +  C_4$
$= K_2 + K_2 + K_2 = K_3 \cdot K_2$,
the graph to the right has the additive prime decomposition
$P_7  +  K_1  +  K_2  +  K_2 + K_2$. In this case, the additive prime
decomposition is unique. In all cases we have seen so far, the composition
is unique, indicating that $(\mathbb{G}, + )$ could be a unique 
factorization monoid.  \label{factorization}
}
\end{figure}

Finally, lets look at the relation between the Fredholm characteristic:

\begin{coro}
On the sub monoid of graphs with even Euler characteristic, the
Fredholm characteristic is multiplicative: 
Then $\psi(G + H ) = \psi(G) \psi(H)$. This happens especially on spheres
so that on that group $\psi$ becomes a character when extended to negative 
graphs with $\psi(-G)=1/\psi(G)$. 
\end{coro}
\begin{proof}
Check cases: \\
If $\psi(G)=\psi(H)=1$, then $\psi(G*H)=1$.  \\
If $\psi(G)=\psi(H)=-1$, then $\psi(G*H)=1$. \\
If $\psi(G)=-\psi(H)=1$, then $\psi(G*H)=-1$. 
\end{proof}

\begin{coro}
In the sub-monoid of graphs with odd Euler characteristic, then $\psi(G+H) = -1$.
This implies that if $\psi(G)=1$, then $G$ is prime in that monoid. 
\end{coro}
\begin{proof}
Check cases:  \\
If $\psi(G)=\psi(H)=1$, then $\psi(G+H)=-1$.  \\
If $\psi(G)=\psi(H)=-1$, then $\psi(G+H)=-1$. \\
If $\psi(G)=-\psi(H)=1$, then $\psi(G+H)=-1$. 
\end{proof}

{\bf Examples.} \\
{\bf 1)} Among complete graphs only $K_1$ has Fredholm characteristic $1$. All others
have Fredholm characteristic $-1$ and can therefore not be prime. \\
{\bf 2)} There is a two dimensional graph $G$ with the topology of the projective plane. 
Its $f$-vector is $(15,42,28)$ so that $\psi(G)=1$. Its Euler characteristic is $1$
too. Therefore, $G$ must be prime. \\
{\bf 3)} A tree with an even number of edges must be prime in the monoid of 
odd Euler characteristic graphs.  \\

When mentioning the arithmetic of graphs to An Huang, he 
asked me about relations of the spectrum of the join and the individual 
components. Here are three remarks: 

\begin{lemma}
If $G=H+K$ then there is an eigenvalue $|V(H)| + |V(K)|$ of the scalar Laplacian $L$. 
Consequently, any graph $G$ for which there is no eigenvalue $v_0(G)$ for $L$
is prime. 
\end{lemma}
\begin{proof}
If $m=|V(H)|$ and $n=|V(K)|$, then the vector which 
is constant $n$ on $H$ and constant $m$ on $K$ is an eigenvector
to the Laplacian $L$ with eigenvalue $n+m$. The eigenvector is
perpendicular to the constant. 
\end{proof} 

\begin{coro}
The graph $n G= G+G+ \cdots + G$ has the eigenvalue $n |V(G)|$ with multiplicity
at least $n-1$. 
\end{coro} 
\begin{proof} 
The multiplicity follows because we can write the sum in $n-1$ different ways 
as a sum of two graphs $A+B$. The proof above determine then the eigenvector. 
\end{proof}

The example $K_n=K_1+K_1+ \dots K_1$, where $n$ appears with multiplicity $n-1$ is an extreme case. 
An other case is the $(n-1)$-dimensional cross polytop $n P_2 = P_2 + ... + P_2$ which has the eigenvalue
$2n$ with multiplicity $n-1$.  \\

Of interest is also the smallest non-zero eigenvalue $\lambda_2$ of the Laplacian. 

\begin{lemma}
$\lambda_2(G+H) = {\rm min}(|V(H)|,|V(K)|) + {\rm min}(\lambda_2(G),\lambda_2(H)$. 
\end{lemma}
\begin{proof}
The Courant-Fischer formula tells 
$\lambda_2 = {\rm inf}_{v \cdot 1 = 0} (v,Lv)/(v,v)$. 
Let $v_2(H)$ be the eigenvector of $L(H)$ with eigenvalue $\lambda_2(H)$
and $v_2(G)$ the eigenvector of $L(H)$ with eigenvalue $\lambda_2(G)$. 
Assume they are normalized. 
Extend $v_2(G)$ and $v_2(H)$ onto $G+H$ by putting $0$ on the other part. 
They still have norm $1$ on the product space. They are also perpendicular to
the constant vector. Now $(v_2(H),L v_2(H))=\lambda_2(H)+m$ and
$(v_2(G),L v_2(G)) = \lambda_2(G) + n$. This shows
$$\lambda_2(G+H) \leq  {\rm min}(|V(H)|,|V(K)|) + {\rm min}(\lambda_2(G),\lambda_2(H) \; . $$
On the other hand, since $G+H$ can be obtained from the disjoint union of 
$G$ and $P_n$ (the graph without edges), by adding edges, we have
$\lambda_i(G) + n \leq \lambda_i(G+H)$. 
Similarly $\lambda_i(H) + m \leq  \lambda_i(G+H)$.  
See Corollary 4.4.2. in \cite{Spielman2015}. 
\end{proof} 

Here is an other spectral result. It deals with the highest form Laplacian. 
Let $D=d+d^*$ be the Dirac operator of a simplicial complex. The form Laplacian $L=(d+d^*)^2$ 
splits into blocks called form Laplacians $L_k$, the restriction of $L$ on 
discrete $k$-forms. The nullity of $L_k$ is the $k$'th 
Betti number $b_k$ \cite{DiracKnill,knillmckeansinger}. We have already seen 
that if we have two graphs $G,H$, then the volume of $G+H$ is the product of the volumes
of $G$ and volumes of $H$. Lets write $L_v(G)$ for $L_{{\rm dim}(G)}(G)$ for the volume
Laplacian, the Laplacian belonging to the largest dimension. Now if $G$ has volume $m$
and $H$ has volume $n$ then we can look at the volume eigenvalues 
$\lambda_1, \dots, \lambda_n$ of $L_v(G)$ and the volume eigenvalues 
$\mu_1, \dots , \mu_m$ of $L_v(H)$. What are the volume eigenvalues of $H+G$? 

\begin{lemma}
For two arbitrary simplicial complexes $G,H$, 
the volume eigenvalues of $G+H$ are given by 
$\lambda+\mu$, where $\lambda$ runs over the volume eigenvalues of $G$
and $\mu$ runs over the volume eigenvalues of $H$. 
\end{lemma}
\begin{proof}
If $U$ is the volume Laplacian of $G$ and $K$ is the volume Laplacian of $H$.
Given $U v = \lambda v$ and $V w = \mu w$. The volume Laplacian $W$ of $G+H$
is a $(nm) \times (nm)$ matrix because the volumes of $G$ and $K$ have multiplied. 
Define on the facet $xy$ of $G+H$ joining the facets $x$ and $y$ of $G$ and $H$ the 
value $f_{xy} = v_x w_y$. Now $W f = \lambda f + \mu f$. 
\end{proof}

\section{A ring of networks}

The group $(X,+,0)$ with Zykov addition $+$ (graph join) has a compatible multiplication. 
We formulate it only in the graph case, meaning the case of simplicial complexes which are 
Whitney complexes of graphs, even so it would also work for general simplicial complexes.  \\

Given two graphs $G=(V,E), H=(W,F)$, define the graph $G \cdot H=(V \times W, Q)$,
where the vertex set $V \times W$ is the Cartesian product of the two vertex sets and where
the edge set $Q$ consists of all pairs 
$$  \{ ((a,b),(c,d) \; | \; (a,c) \in E \; {\rm or} \; (b,d) \in F \} \; . $$

This means that we connect $(a,b)$ with any $(c,x)$ if $(a,c) \in E$ 
and with any $(a,b)$ with $(x,d)$ if $(b,d) \in F$. This product is obviously
associative and commutative. 

\begin{lemma}
This product is compatible with the join addition $0$ in the sense that the 
distributivity law $G \cdot (H + K) = G \cdot H + G \cdot K$ holds. 
\end{lemma}
\begin{proof}
Like in the case of the ring of integers, spacial geometric insight can help by 
placing the graphs H,K and H+K into one coordinate axes and G onto the other. Now both 
$G \cdot H + G \cdot K and G \cdot (H+K)$ have as the vertex set the product sets of the vertices. 
As connections in a sum H+K consist of three types, connections within H, 
connections within K and any possible connection between H and K, two points 
$(a,b), (c,d)$ are connected if either $(a,c)$ is an edge in $H$, or 
$(b,d)$ is an edge in K or then if either $a,c$ or $b,d$ are in different graphs. 
\end{proof}

{\bf Examples}: \\
All equality signs mean here ``graph isomorphic".\\
{\bf a)} $P_n \cdot P_m = P_{nm}$ \\
{\bf b)} $K_n + K_m = K_{n+m}$ \\
{\bf c)} $K_n \cdot K_m = K_{n m}$ \\
{\bf d)} $1 + G$  is cone over $G$ \\
{\bf e)} $1+ C_n = W_n$ wheel graph \\
{\bf f)} $P_2 + G$ suspension of $G$ \\
{\bf g)} $P_2 + C_4=P_2+P_2+P_2=K_3 \cdot P_2$ octahedron graph: Oct \\
{\bf h)} $K_{3,3} = P_3 + P_3 = K_2 \cdot P_3$ utility graph \\
{\bf i)} $K_{n,m} = P_n + P_m$ complete bipartite graph \\
{\bf j)} $1+P_n = S_n$ star graph \\
{\bf k)} $1 + S_3 = 1+1+P_n = K_2+P_n$ windmill graph  \\
{\bf i)} $K_2 \cdot C_4 = C_4 + C_4$ three sphere \\
{\bf j)} $K_3 \cdot {\rm Oct} = {\rm Oct} + {\rm Oct} + {\rm Oct}$ is an 7-dimensional sphere  \\
{\bf k)} $K_{n,m}+K_{n,m} = K_2 \cdot K_{n,m} = K_2 (P_n + P_m) = K_{n,n} + K_{m,m}$  \\
{\bf l)} $W_4 + P_2 = 1+C_4 +1+C_4 =K_2 + C_4 + C_4$.

\begin{figure}[!htpb]
\scalebox{1.0}{\includegraphics{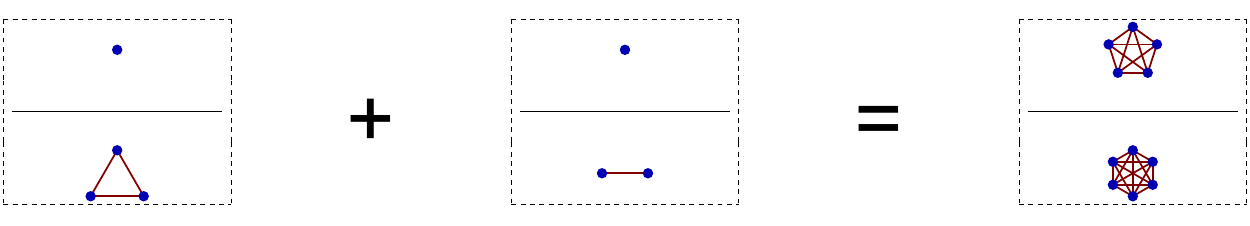}}
\caption{
Working in the field generated by the ring of networks is 
analogues to work with fractions. In school arithmetic, we start with
the monoid $\mathbb{N}=\{0,1,2,3, \dots \}$, then introduce the 
integers $\mathbb{Z}$, then build multiplication leading to the
ring of integers and finally produce the field of fractions. 
Here we see the identity $1/3 + 1/2 = 5/6$ in network arithmetic. 
All graphs involved are here complete graphs where arithmetic is equivalent
to the arithmetic in the rationals $\mathbb{Q}$. 
\label{fraction1}
}
\end{figure}

\begin{figure}[!htpb]
\scalebox{1.0}{\includegraphics{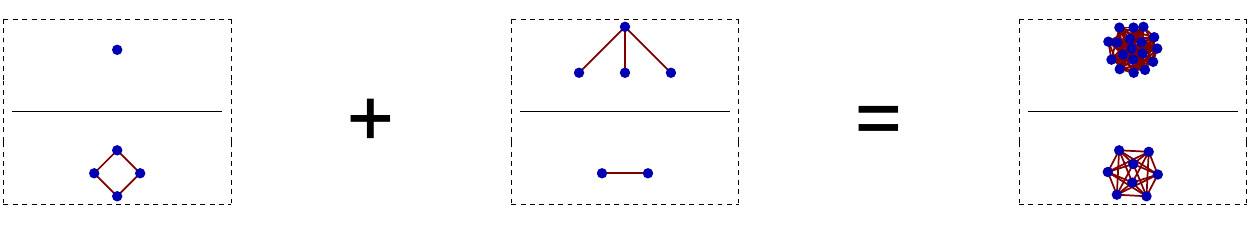}}
\caption{
While the field of networks contains the field $\mathbb{Q}$ 
as $K_n+K_m=K_{n+m}$ and $K_n \cdot K_m = K_{n m}$, we can 
work with more general graphs. Here we see the result of
the addition of two fractions: $1/C_4 + S_3/K_2$, where $C_4$
is the circular graph and $S_3$ is the star graph. The addition
is done as in school arithmetic. We get $K_1 \cdot K_2 + S_3 \cdot C_4
=K_2+S_3 \cdot C_4$ in the nominator and $C_4 \cdot K_2$ in the 
denominator.
\label{fraction2}
}
\end{figure}

One of the consequences of the fact that we know the maximal eigenvalue of $G+H$ is
that we can get the maximal eigenvalue of some graphs immediately.

\begin{coro}
For $n>1$, the maximal eigenvalue of $K_n \cdot G$ is $n |V(G)|$. 
\end{coro}

This implies for example that the maximal eigenvalue
of $K_n$ is $n$ for $n>1$ and the maximal eigenvalue for $K_{n,m}$
is $n+m$ for $n,m \geq 1$. This is extremal as $\lambda_k \leq |V(G)|$ for
all eigenvalue. 
It also shows that because $K_n P_2$ is the $n-1$ dimensional sphere
(cross polytop), that the maximal eigenvalue is $2n$. Indeed, the octahedron 
$K_3 \cdot P_2$ for example has maximal eigenvalue $6$ and the $16$-cell
$K_4 \cdot P_2$ has maximal eigenvalue $8$ with multiplicity at least $3$. 

\begin{figure}[!htpb]
\scalebox{0.2}{\includegraphics{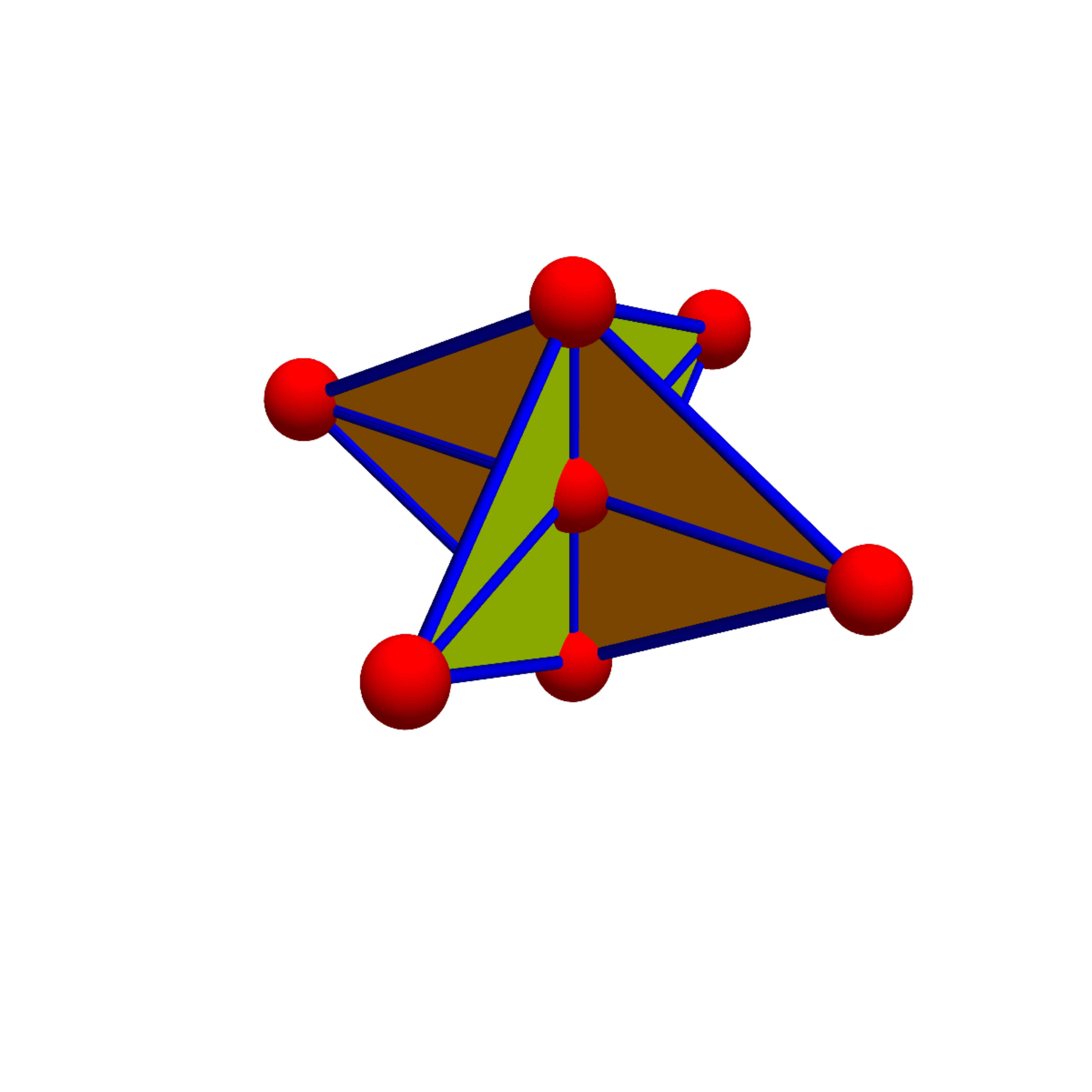}}
\caption{
The graph $1+P_4 + P_2 = S_4 + P_2$ is the suspension of the star graph $S_4=1 + P_4$. 
As a non-prime graph, its maximal eigenvalue is the number of vertices, 
which is $7$. Since $P_4$ has only $0$ eigenvalues, the graph $S_4=1+P_4$ has $\lambda_2=1$.
According to the spectral lemma for $\lambda_2$ above, 
we know that $P_2 + S_4$ has eigenvalue $2+0+1=3$. 
\label{example1}
}
\end{figure}

\begin{figure}[!htpb]
\scalebox{0.2}{\includegraphics{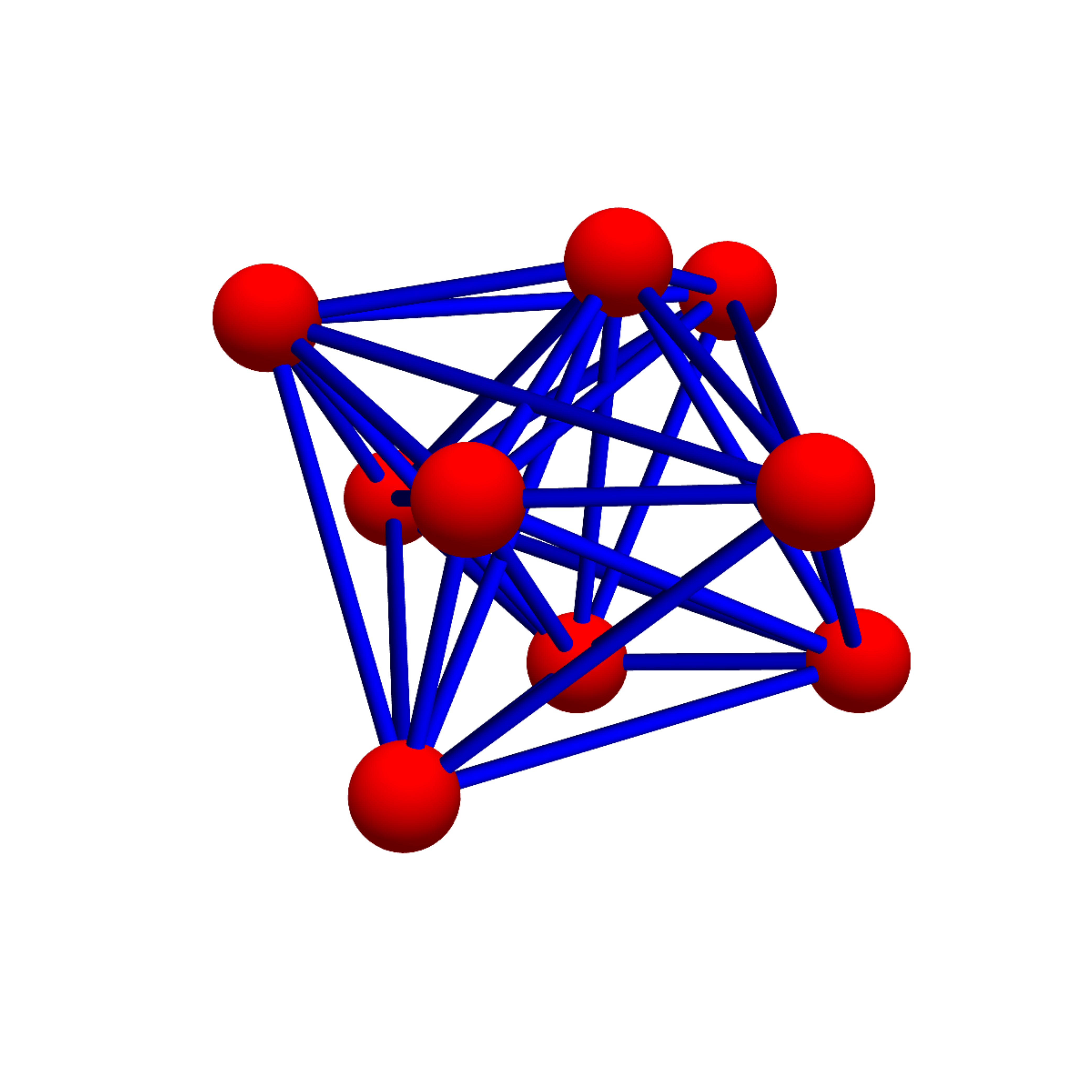}}
\caption{
The graph $K_3 \cdot L_2 = K_3 \cdot (1+ P_2) = K_3 + K_3 \cdot P_2 
= K_3 + {\rm Oct} = 1+1+1+{\rm Oct}$  is the sum of the triangle and the octahedron but it is also $3$ times the 
linear graph $L_2$ of length $2$. It is also the suspension of the cone of the octahedron or
the cone of a 3 sphere. As the graph is not an additive prime, its maximal
eigenvalue is the number of vertices, which is $9$. The multiplicity result above only 
assures the maximal eigenvalue to appear with multiplicity $2$. The eigenvalue $9$ actually
appears with multiplicity $5$. 
\label{example2}
}
\end{figure}

\section{Questions} 

Lets finish with some questions. \\

{\bf A)} We still have not identified the off diagonal terms $(1+A(G'))^{-1}_{xy}$ 
in topological terms, where $G'$ is the connection graph of a complex. 
Having the diagonal terms related to natural sphere index functionals, which are very 
closely related to Morse theory, it is likely that also the off diagonal
terms refer to some interesting topological or dynamical notion.  \\

{\bf B)} Since we have identified two characters $i(G) = 1-\chi(G)$ and $\psi(G)$ 
on the sphere group (a homomorphism from the sphere group to the group 
$\{ |z| = 1 \} | \; z \in \mathbb{C} \}$), 
it is natural to ask whether there is more representation theory of the sphere group for
which the character are classical characters, represented by a trace of a matrix
associated to a sphere. We especially would like to know the collection
of all characters. Since the sphere group is discrete, one could take any distance
(like the minimal number of vertices or edges to be modified to get from one to an isomorphic
image of the other). This is then a topological group which has a Haar measure
as discreteness assures that it is locally compact and such that the Haar measure, the 
counting measure. There are many characters on the sphere group like assigning 
real numbers $k(x)$ to every vertex $x \in V(G)$ of a graph $G$ and $-k(x)$ to vertices
of $-G$. A character is then defined as
$$  X(G-H) = \prod_{x \in V(G)} e^{2\pi i k(x)}/\prod_{x \in V(H)} e^{2\pi i k(x)}  \; . $$
Since adding two graphs takes the union of the vertices, this is multiplicative for the additive
group of spheres. But as both $i(G)$ and $\psi(G)$ tap into the structure 
of the simplices in the graph, it is likely that there are more interesting 
examples of characters. \\

{\bf C)} The join and multiplication in which the addition is not extended yet, forms a 
commutative semiring of type $(2,2,0,0)$: both addition and multiplication are commutative
monoids, distributivity holds and $a \cdot 0 = 0$.
The question of unique prime decomposition in the additive and multiplicative 
monoids of abstract finite simplicial complexes is interesting. Maybe the question is more
approachable on submonoids like the sphere monoid. On the multiplicative submonoid of complete graphs,
the prime factorization is the same than on the multiplicative group of natural numbers without $0$. \\

{\bf D)} We would also like to be able to compute in the field of fractions defined by the ring and
then do a completion solve equations like $G+G=C_5$ leading to $G=C_5/2$ or solving $G \cdot G=K_5$ with 
solution $\sqrt{K_5}$. Both equations can not be solved in the ring 
because $C_5$ is an additive prime (every cyclic graph $C_n$ with $n>4$ is an additive prime) 
and $K_5$ is a multiplicative prime because every graph $K_p$ with $p$ vertices
is a multiplicative prime if $p$ is a rational prime. \\

{\bf E)} There are interesting questions related to the number theoretically defined
prime graphs $G_n$ and prime connection graphs $H_n$. 
One can for example try to estimate the Green function values 
from above. The topology of $G_n$ is linked the Riemann hypothesis as it is directly 
linked to the growth of the Euler characteristic of $G_n$ by the Mertens connection \cite{CountingAndCohomology}.
The Fredholm characteristic of truncations of truncated prime graphs on the other hand gives
a signature, comparing the parity of square free integers $\leq n$ which have 
an even number of prime factors. Such integers correspond to odd-dimensional simplices in the complex.
The Green function values for prime graphs, the matrix entries of $(1+A(H_n))^{-1}_{xy}$ could contain
some interesting number theory. 

\bibliographystyle{plain}

\end{document}